\documentclass[12pt,a4paper]{article}
 \pdfoutput=1

\usepackage{amsmath}
\usepackage{amsfonts}
\usepackage{amsthm}
\usepackage{amssymb}
\usepackage{url}
\usepackage[dvips]{graphicx}
\usepackage{psfrag, color}

\newtheorem{thm}{Theorem}[section]
\newtheorem{lem}{Lemma}[section]

\newtheorem{cor}[lem]{Corollary}
\newtheorem{prop}[thm]{Proposition}

\newtheorem{rmk}{Remark}[section]

\theoremstyle{remark}

\newcommand{\diag}{\operatorname{diag}}
\newcommand{\Z}{\mathbb{Z}}

\newcommand{\R}{\mathbb{R}}
\newcommand{\NN}{\mathbb{N}}

\newcommand{\sM}{\mathcal{M}}

\newcommand{\mF}{\ensuremath{\mathcal{F}}}

\newcommand{\mM}{\ensuremath{\mathcal{M}}}

\newcommand{\T}{\ensuremath{\mathbb{T}}}

\newcommand{\Phg}{\Phi_{\mathrm{glob}}}
\newcommand{\Phl}{\Phi_{\mathrm{loc}}}

\def\to{\longrightarrow}

\def\vk{\vec k}

\def\bdef{\begin{definition}}
\def\endef{\end{definition}}
\def\bthm{\begin{thm}}
\def\ethm{\end{thm}}
\def\blm{\begin{lemma}}
\def\elm{\end{lemma}}
\def\brm{\begin{rmk}}
\def\erm{\end{rmk}}
\def\bprop{\begin{proposition}}
\def\eprop{\end{proposition}}
\def\bcor{\begin{cor}}
\def\ecor{\end{cor}}
\def\be{\begin{eqnarray}}
\def\ee{\end{eqnarray}}
\def\beal{\begin{aligned}}
\def\enal{\end{aligned}}

\def\al{\alpha}
\def\sg{\sigma}
\def\bt{\beta}
\def\eps{\varepsilon}
\def\phi{\varphi}

\def\R{\mathbb R}
\def\C{\mathbb C}

\def\T{\mathbb T}

\def\Z{\mathbb Z}

\def\gm{\gamma}
\def\Gm{\Gamma}

\def\th{\theta}
\def\dt{\delta}
\def\lb{\lambda}
\def\cS{\mathcal S}

\def\be {\begin{equation}}
\def\ee {\end{equation}}
\def\bdef{\begin{definition}}
\def\endef{\end{definition}}
\def\blm{\begin{lem}}
\def\elm{\end{lem}}
\def\beal{\begin{aligned}}
\def\enal{\end{aligned}}

\newtheorem{definition}{Definition}
\begin{document}

\title{Normally hyperbolic invariant
manifolds near strong double  resonance}

\maketitle

\begin{abstract}
  In the present paper we consider a generic perturbation of a nearly
  integrable system of $n$ and a half degrees of freedom
  \be \label{main-Hamiltonian} H_\eps(\th,p,t)=H_0(p)+\eps
  H_1(\th,p,t),\quad \th\in \T^n,\ p\in B^n, \ t\in \T=\R/\Z, \ee with
  a strictly convex $H_0$. For $n=2$ we show that at a strong double
  resonance there exist $3$-dimensional normally hyperbolic invariant
  cylinders going across. This is somewhat unexpected, because at a
  strong double resonance dynamics can be split into one dimensional
  fast motion and two dimensional slow motion. Slow motions are
  described by a mechanical system on a two-torus, which are
  generically chaotic.

  The construction of invariant cylinders involves finitely smooth
  normal forms, analysis of local transition maps near singular points
  by means of Shilnikov's boundary value problem, and
  Conley--McGehee's isolating block.
\end{abstract}

\vskip 0.2in

\author{\qquad \qquad \qquad \qquad V. Kaloshin\footnote{University of
    Maryland at College Park (\url{vadim.kaloshin@gmail.com})} \qquad
  \qquad K. Zhang\footnote{University of Toronto
    (\url{kzhang@math.utoronto.ca})}}

\vskip 0.3in

\markboth{V. Kaloshin, K. Zhang}{Normally hyperbolic invariant
  cylinders at a double resonance}

\section{Introduction}
\label{sec:intro}

Consider the near integrable system from the abstract with $B^n
\subset \R^n$ --- the unit ball around $0$, $\T^n$ --- being the
$n$-torus, and $\T$ --- the unit circle, respectively. Notice that for
$\eps=0$ action component $p$ stays constant. For completely
integrable systems coordinates of this form exist and called {\it
  action-angle}. The famous question, called {\it Arnold diffusion},
is the following

\vskip 0.1in

{\bf Conjecture} \cite{Ar1,Ar2}\ {\it For any two points $p',p''\in
  B^2$ on the connected level hyper-surface of $H_0$ in the action
  space there exist orbits connecting an arbitrary small neighborhood
  of the torus $p=p'$ with an arbitrary small neighborhood of the
  torus $p=p''$, provided that $\eps \ne 0$ is sufficiently small and
  that $H_1$ is generic.}

\vskip 0.1in

A proof of this conjecture for $n=2$ is announced by Mather \cite{Ma}.

The classical way to approach this problem is to consider a finite
collection of resonances $\Gm_1, \ \Gm_2, \dots, \Gm_{N+1}\subset B^2$
so that $\Gm_1$ intersects a neighborhood of $p'$, $\Gm_{N+1}$ intersects
a neighborhood of $p''$,
and $\Gm_{j+1}$ intersects $\Gm_j$ for
$j=1,\dots,N$ and diffuse along them. This naive idea faces
difficulties at various levels.

Fix an integer relations $\vec k_1 \cdot \partial_p H_0 +k_0=0$ with
$\vec k=(\vec k_1,k_0) \in (\Z^2 \setminus 0) \times \Z$ and $\cdot$
being the inner product define one-dimensional resonances.  Under the
condition that the Hessian of $H_0$ is non-degenerate, each resonance
defines a smooth curve embedded into the action space $\Gm_{\vec k}=\{p\in
B^2:\ \vec k_1 \cdot \partial_p H_0 +k_0=0\} \footnote{such a curve
  might be empty}.$ Such a curve is called a {\it resonance}. If one
intersects resonances corresponding to two linearly independent $\vk$
and $\vk'$ we get isolated points. In the case when both $\vec k$ and
$\vec k'$ are relatively small, i.e. $|\vec k|,\ |\vec k'|<K$ for some
$K>1$.  We call such an intersection {\it a $K$-strong double resonance} or
simply a strong double resonance (if using $K$ is redundant); see Figure \ref{res-net}.
So far only examples of strong double resonances have been studied
(see \cite{Bs,KLS,KS,KZZ}).

\subsection{Diffusion along single resonances by means of crumpled
normally hyperbolic cylinders}
Fix one resonance $\Gm$. In \cite{BKZ} we prove that depending on a
generic $H_1$ ({\it but not on $\eps$!}) there are a finite number of
punctures of $\Gm$. In other words, there is $K=K(H_1)>0$ such that
way from $\eps^{1/6}$-neighborhood of any $K$-strong double resonance
there are diffusing orbits
along $\Gm$. Moreover, these diffusing orbits
are constructed in two steps:


\begin{itemize}
\item Construct invariant normally hyperbolic invariant cylinders
  (NHIC) ``connecting'' a $\eps^{1/6}$-neighbor\-hood of one
  $K$-strong double resonance with a $\eps^{1/6}$-neighbor\-hood of
  the next one on $\Gm$.

\item Construct orbits diffusing along these cylinders, which is done
  using Mather variational method \cite{Be, CY1, CY2}.
\end{itemize}

It turns out that these cylinders are crumpled in the sense that
its regularity blows up as $\eps \to 0$. See Figure \ref{fig:crumpled}.
Existence of crumpled NHICs is the new phenomenon, discovered in
\cite{BKZ}. In spite of this irregularity, one can use them for
diffusion.

\begin{figure}[t]
  \centering
  \includegraphics[width=4in]{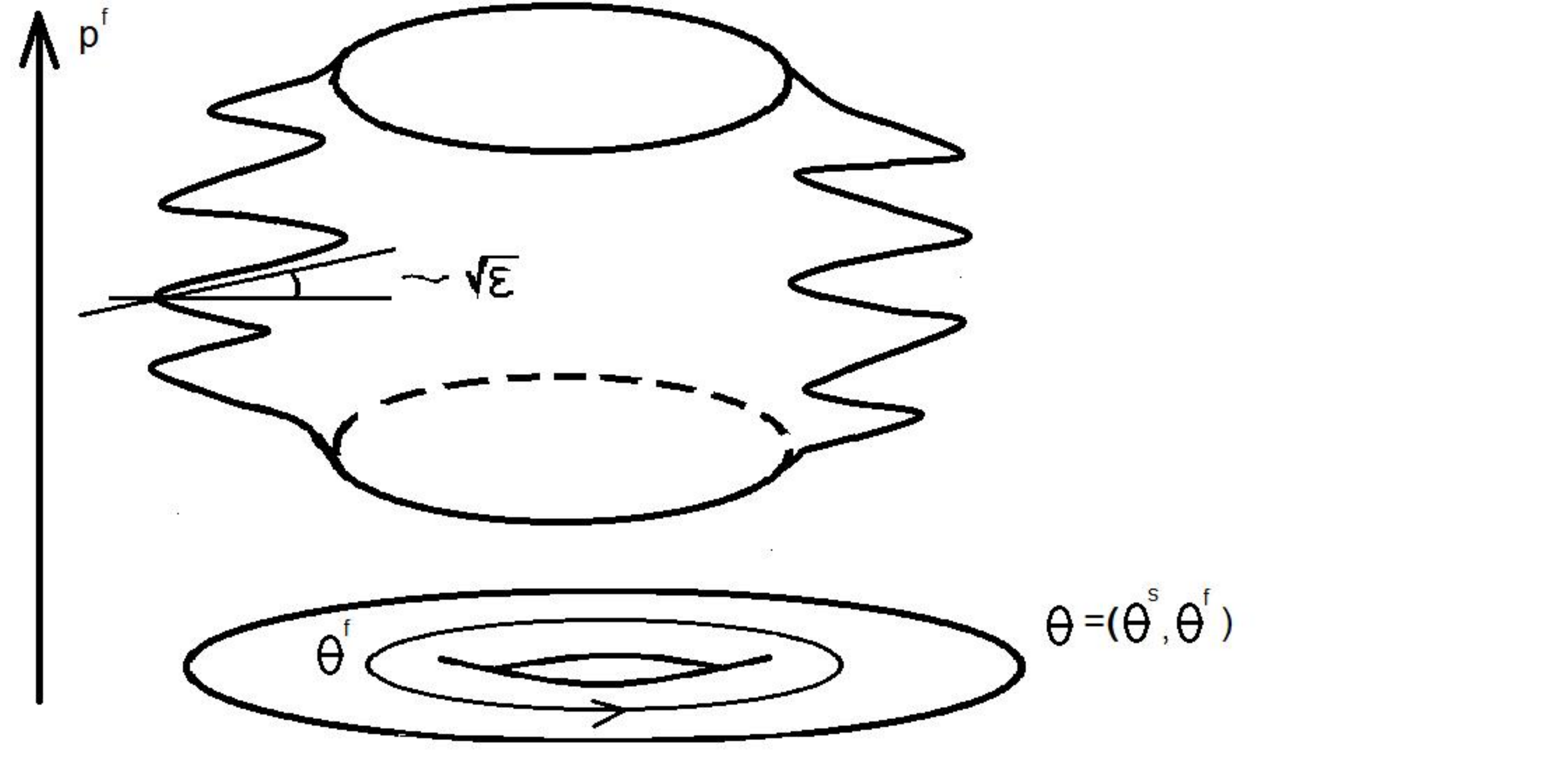}
  \caption{Crumpled Cylinders}
  \label{fig:crumpled}
\end{figure}

The main topic of the present paper is {\it how to diffuse across a
  strong double resonance}. We propose a heuristic description and
prove existence of underlying normally hyperbolic invariant manifolds
(NHIMs) for it.

\subsection{Strong double resonances and slow mechanical systems}

We fix two independent resonant lines $\Gm, \ \Gm'$ and a strong
double resonance $p_0\in \Gm \cap \Gm'\subset B^2$. Then the standard
averaging along the one-dimensional fast direction gives rise to a
slow mechanical system $H^s=K(I^s)-U(\th^s)$
of two degrees of freedom, where $\th^s\in \T^s$ and
$I^s$ is a rescaled conjugate variable. Namely, in $O(\sqrt \eps)$-neighborhood of $p_0$,  after a canonical coordinate change and rescaling the action variables, the flow of the Hamiltonian $H_\epsilon$ is conjugate to that of 
\[
 c_0/\sqrt{\epsilon} + \sqrt{\epsilon}(K(I^s)-U(\th^s)) + O(\epsilon).
\]
Precise definitions of $c_0$, $\th^s, I^s$ are in Section \ref{slow-fast}.
The slow kinetic energy $K$ and the slow potential energy $U$ are
defined in (\ref{slow-kinetic}--\ref{slow-potential}), respectively.

\begin{figure}[t]
  \begin{center}
    \includegraphics[width=3in]{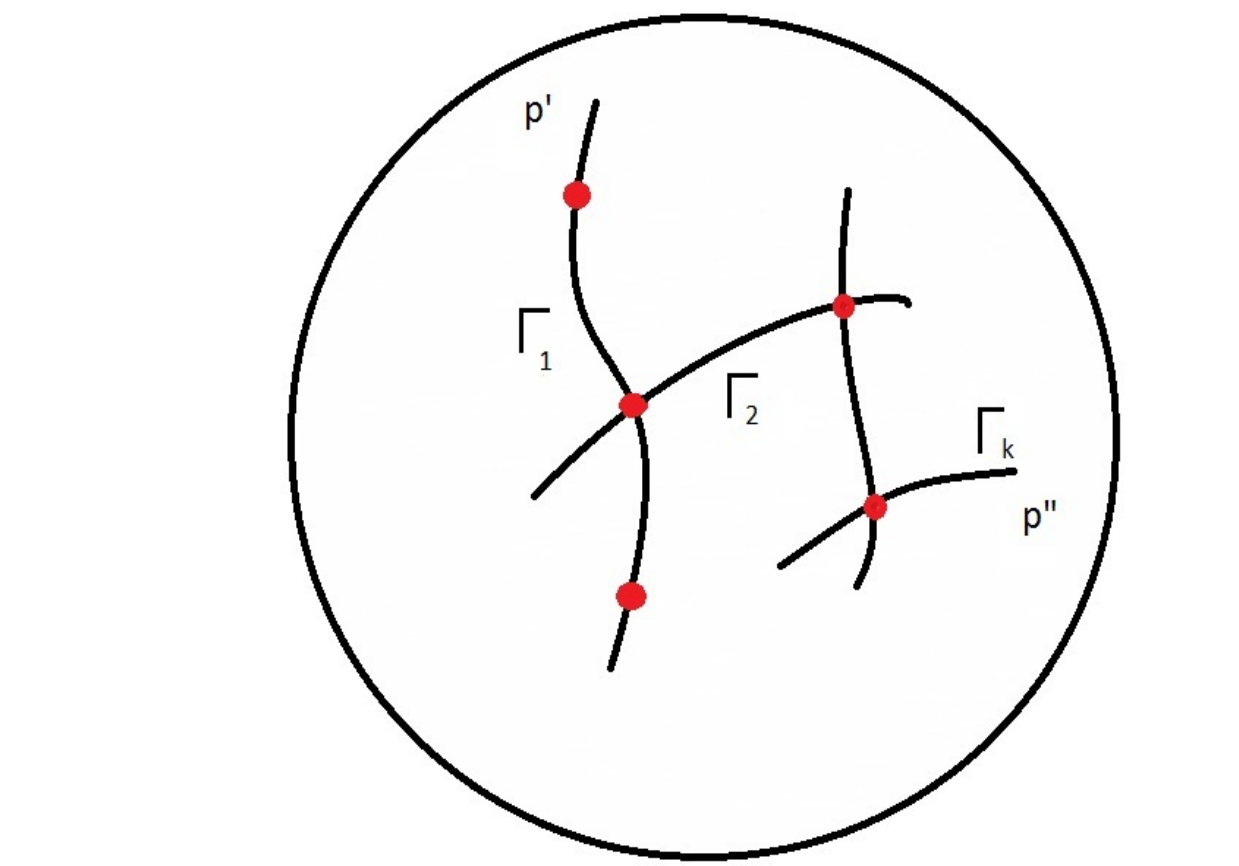}
  \end{center}
  \caption{Resonant net}
  \label{res-net}
\end{figure}

From now on we analyze the slow mechanical system $H^s=K(I^s)-U(\th^s)$.
Denote by
$$
\cS_E=\{(\th^s,I^s):\ H^s=E\}
$$
an energy surface. Without loss of generality assume that the minimum
$\min_{\th^s} U(\th^s)=0$, it is unique, and occurs at
$\th^s=0$. According to the Mapertuis principle for a positive energy
$E>0$ orbits of $H^s$ restricted to $\cS_E$ are reparametrized
geodesics of the Jacobi metric
\be \label{Jacobi-metric}
\rho_E(\th) = 2(E+U(\th))\ K.
\ee

Notice that the resonance $\Gm \subset B^2$ (resp. $\Gm' \subset B^2$)
induces an integer homology class $h$ (resp. $h'$) on $\T^s$, i.e.
$h$ (resp. $h'$) $\in H_1(\T^s,\Z)$. Denote by
$\gm_h^E$ (resp. $\gm_{h'}^E$) a minimal geodesic of $\rho_E$ in the
homology class $h$ (resp. $h'$).
For example, if $H_\eps(\th,p,t) = \dfrac 12 p^2+\eps H_1(\th,p,t)$,
  $\Gm=\{p:\ \partial_{p_2}H_0=p_2=0\}$ (resp.
  $\Gm'=\{p:\ \partial_{p_1}H_0=p_1=0\}$). Then on the slow torus
  $\T^s\ni (p_1,p_2)$ we have $h=(1,0)$ (resp. $h'=(0,1)$).
On the unit energy surface $S_1$ the strong resonance occurs at
$p_0=(0,0,1)\in \Gm\cap \Gm'$.

\subsection{Two types of NHIMs at a strong double resonance}

Notice that diffusing along $\Gm$ for the Hamiltonian $H$ corresponds
to changing slow energy $E$ of $H^s$ along the homology class $h$. In
particular, we need to get across zero energy. However,
$\cS_0=\{(\th^s,I^s):\ H^s=0\}$ is the critical energy surface,
namely, the Jacobi metric is degenerate at the origin.

There are
{\it at least two special
}\footnote{in order to find these two homology classes one needs to
  find minimal geodesics $\gm_0^h$ in each integer homology class and
  minimize its length over all $h\in H_1(\T^s,\Z)$.  Then pick two
  Jacobi-shortest ones.} integer homology classes $h_1,h_2\in
H_1(\T^s,\Z)$ such that minimal geodesics of $g_0$ are
non-self-intersecting. 
For $i=1,2$, we denote by $\gamma_E^{-h_i}$ the curve obtained by the time reversal $I^s\to -I^s$ and $t\to -t$. Then   
  \be \label{simple-NHIC}
  \beal \textit{
    the union   of }
    \gm_E^{\pm h_i}\textit{ over }0 \le E \le E_0\textit{
  is contained in a }\\
  C^1\textit{ smooth two-dimensional NHIM }\mathcal M_{h_i}. \qquad \quad
  \enal
  \ee

This imply that the original Hamiltonian system
$H_\eps$ also has two three-dimensional NHIM $C^1$-close to $\mathcal
M_{h_1}$ and $\mathcal M_{h_2}$. By the reason to be clear later we call
such cylinders {\it simple loop cylinders}.

Moreover, for small $E_0>0$ and all energies
$E_0<E<E_0^{-1}$ except finitely many $\{E_j\}_{j=1}^N \subset [E_0,E_0^{-1}],
E_j<E_{j+1}, j=1,\dots,N-1$ we show that for each $1\le j\le N$
a proper union $\mathcal M^h_j$ of $\gm_E^{h}$ over $E_j<E<E_{j+1}$
form $C^1$ smooth NHIC. This imply that the original Hamiltonian system
$H_\eps$ also has a NHIC $C^1$-close to $\mathcal M^h$. See the Appendix
for more details.

\subsubsection{Non-simple figure $8$ loops}\label{figure8-case-only}
If minimal geodesics of $\rho_0$ are self-intersecting the situation was
described by Mather \cite{Ma4}. Generically $\gm^E_h$ accumulates onto
the union of two simple loops, possibly with multiplicities. More precisely,
given $h\in H_1(\T^s,\Z)$ generically 
there are homology classes
$h_1,\ h_2 \in H_1(\T^s,\Z)$ and integers $n_1,\ n_2\in\Z_+$ such that
the corresponding  minimal geodesics $\gm_0^{h_1}$ and $\gm_0^{h_2}$
are simple and $h=n_1h_1+n_2h_2$. Denote $n=n_1+n_2$. 

For $E>0$, $\gamma^h_E$ has no self intersection. As a consequence, there is a unique way to represent $\gamma_h^0$ as a concatenation of $\gamma^{h_1}_0$ and $\gamma^{h_2}_0$. More precisely, we have the following lemma. 
\begin{lem}\label{homotopy}
  There exists a  sequence $\sigma=(\sigma_1, \cdots, \sigma_n)\in \{1,2\}^n$, unique up to cyclical translation, such that 
  $$\gamma^h_0 = \gamma^{h_{\sigma_1}}_0 * \cdots * \gamma^{h_{\sigma_n}}_0. $$
\end{lem}
Lemma~\ref{homotopy} will be proved in  Appendix~\ref{sec:homology}.  In this case for
small energies this cylinder resembles the figure $8$ and
we call it {\it  two leaf cylinder} and call 
{\it the corresponding $\gamma^h_0$ non-simple}.

We would like to point out that Jean-Pierre Marco \cite{Mar1, Mar2} is
studying similar ideas.

\subsubsection{Kissing cylinders}


If the loop $\gamma^h_0$ is non-simple, then the union $\mathcal M_h^8=\bigcup_{0 \le E \le E_0}\gamma^h_E$ \emph{is not a manifold} at $\gamma_h^0$! 
However, $\mathcal M_{h_1}, \mathcal M_{h_2},$
and $\mathcal M_h^8$ all have a tangency at the origin (see Remark
\ref{kissing-cylinders}).  Moreover, expanding and contracting
directions at the origin of all the three normally hyperbolic
invariant manifolds are parallel to strong unstable and strong stable
directions. Simple dimension consideration makes us believe that for
the original Hamiltonian $H_\eps$ has 
NHIMs $\mathcal M^\eps_{h_1}, \mathcal M^\eps_{h_2}, \mathcal M^8_{h}$
with transversal intersections of invariant manifolds.

\begin{figure}[t]
  \centering
  \includegraphics[width=3in]{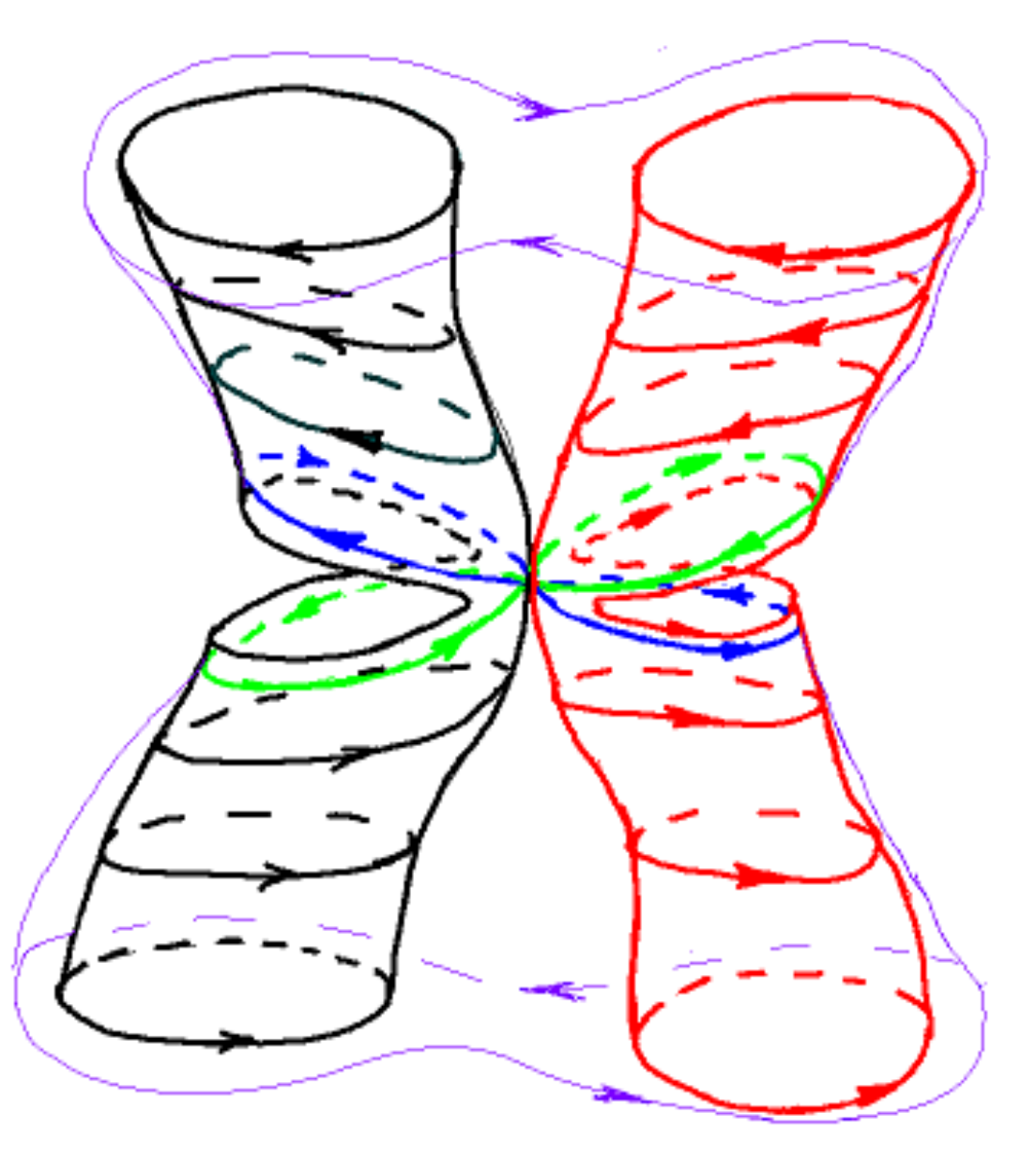}
  \caption{Kissing Cylinders}
  \label{fig:kissing}
\end{figure}

\subsection{An heuristic diffusion through strong double resonances}

We hope the following mechanism of diffusion through double resonance
takes place. As we mentioned above in \cite{BKZ} we show that away
from $\eps^{1/6}$-neighborhood of strong double resonances there are
crumpled NHIC and orbits diffusing along them. It turns out
that in the region where distance to the center of a strong double
resonance is between $[\eps^a,\eps^{1/6}]$ for some
$1/6<a<1/2$ we can slightly modify argument from \cite{BKZ} and show
that the system $H_\eps(\th,p,t)$ has a NHIC. Moreover, this cylinder
is smoothly attached to the crumpled NHICs build in \cite{BKZ}.

\subsubsection{First intermediate zone}

Fix $C=C(H_0,H_1)\gg 1$, but independent of $\eps$.
In the region where distance to the center of a strong double
resonance is between $[C\sqrt \eps,\eps^a]$ we define
{\it a slow-fast mechanical system} and show that it
approximates dynamics of our system $H_\eps(\th,p,t)$
well enough to establish existence of a NHIC. Moreover,
this cylinder is smoothly attached to
the one from the region $[\eps^a,\eps^{1/6}]$.


\subsubsection{Second intermediate zone}
Let $C=E_0^{-1}$.
Consider the region where distance to the center of a strong double
resonance is between $[E_0\sqrt \eps,E_0^{-1}\sqrt \eps]$.
In this regime dynamics is well approximable by
{\it a slow mechanical system}. Thus, we need to study
a mechanical system of two degrees of freedom on an interval
of energy surfaces and its family of minimal geodesics $\{\gm_E^{h}\}_E$
in a given homology class $h\in H_1(\T^s,\Z)$. The left
boundary $E_0 \sqrt \eps$ means that we need to study
a mechanical system for slow energies  $E>E_0$.

Simple analysis, carried out in  Appendix~\ref{sec:intermediate}, shows that
and all energies $E_0<E<E_0^{-1}$ except finitely many
$\{E_j\}_{j=1}^N \subset [E_0,E_0^{-1}],\ E_j<E_{j+1},\
j=1,\dots,N-1$
a proper union $\mathcal M^h_j$ of $\gm_E^{h}$ over $E_j<E<E_{j+1}$
form $C^1$ smooth NHIC. Application of Conley--McGehee's isolating
block implies that the original Hamiltonian system
$H_\eps$ also has a NHIC $C^1$-close to $\mathcal M^h_j$.
Moreover, one can construct diffusing orbits along
$\{\mathcal M^h_j\}_{j=1}^{N-1}$. This part is very much
analogous to the one done in \cite{BKZ}.

\begin{figure}[t]
  \centering
  \includegraphics[width=5in]{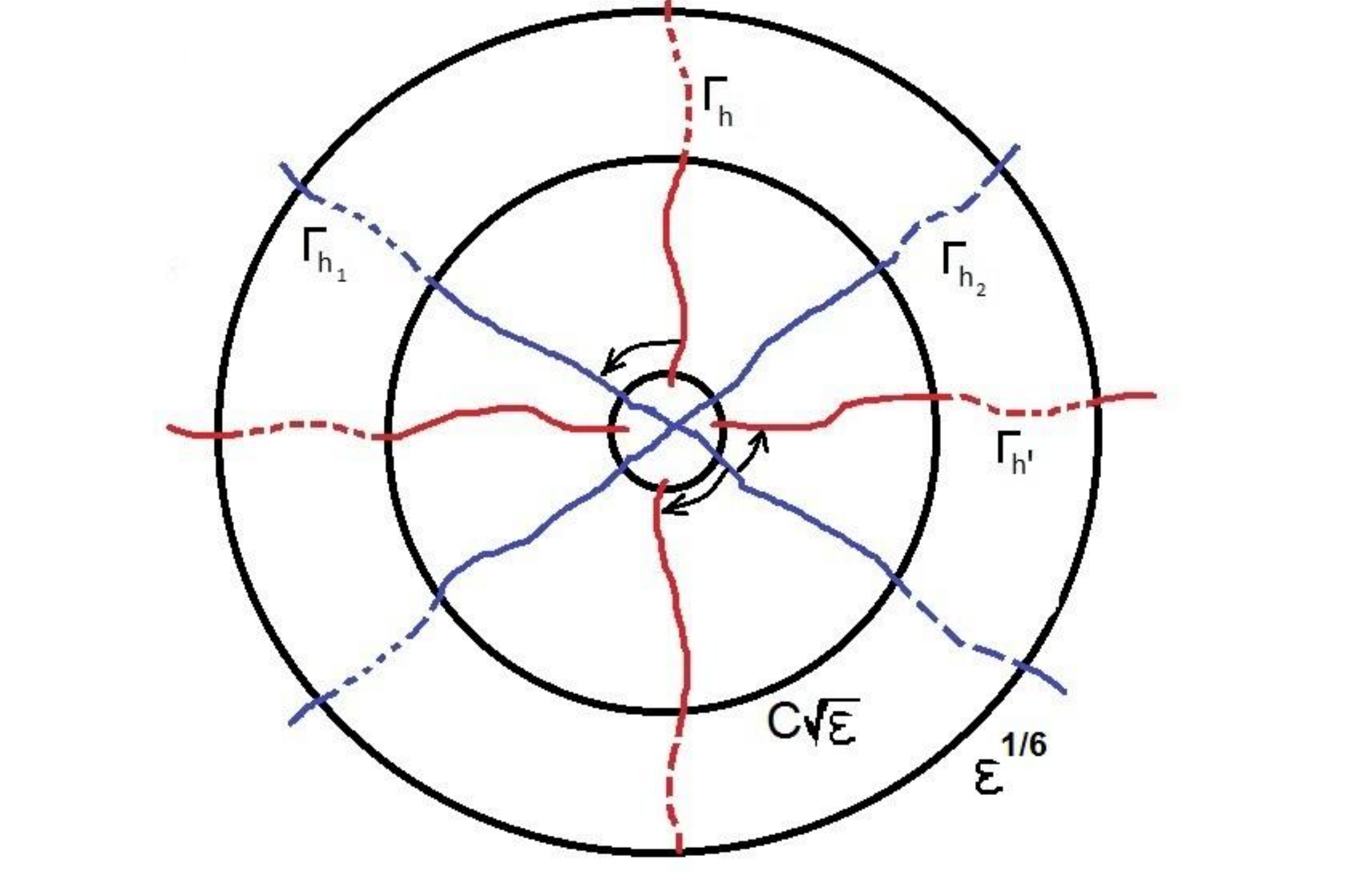}
  \caption{Heuristic description}
  \label{fig:heuritic}
\end{figure}

Now we arrive to slow energy $E_0\sqrt \eps$ near a strong double
resonance and need to consider several cases.
Heuristic description of our mechanism is on Figure \ref{fig:heuritic}.
First, we cross a strong double resonance along $\Gm$.

\subsubsection{Crossing through along a simple loop $\Gm$}
If $\gm^h_0$ is simple or does not pass through the origin at all,
then an orbit enters along a NHIC $\mathcal M^h_1$ and can diffuse
along NHIM $\sM_{h}$ across the center of a strong double resonance
$p_0$ to ``the other side''.

\subsubsection{Crossing through along a non-simple $\Gm$}
If $\gm^h_0$ is non-simple,
i.e the union of two simple loops, then an orbit enters along a NHIC
$\mathcal M^h_1$. As it diffuses toward the center of a strong double
resonance $p_0$ the cylinder $\mathcal M_h^8$ becomes a two leaf
cylinder and its boundary approaches the figure $8$.

For a small enough energy $\dt>0$ of the mechanical system $H^s$ the
two leaf cylinder $\mathcal M_h^8$ is almost tangent to a certain simple
loop NHIC $\mathcal M_{h_i}$. Moreover, near the origin both of normally
hyperbolic invariant manifolds have almost parallel most contract
and expanding directions. As a result there should be orbits
{\it jumping  from the two leaf cylinder $\mathcal M_h^8$ to
a simple loop one $\mathcal M_{h_i}$ from} (\ref{simple-NHIC}).
Then such orbits can cross
the double resonance along $\mathcal M_{h_i}$. After that it
jumps back on the opposite branch of $\mathcal M_h^8$ and diffuse
away
along $\Gm$ as before.

\subsubsection{Turning a corner from $\Gm$ to $\Gm'$}
Now we cross a strong double resonance by entering along $\Gm$ and
exiting along $\Gm'$. As before an orbit enters along a NHIC
$\mathcal M^h_1$ constructed in the second intermediate zone. As
it diffused toward the center of a strong double resonance $p_0$
the cylinder $\mathcal M_h^8$ becomes a two leaf cylinder.

As in the previous case if we diffuse along $\mathcal M_{h}^8$ to
a small enough energy.

\begin{itemize}
\item If $h'$ has a simple loop $\gm^{h'}_0$, then we jump
to $\mathcal M_{h'}$ directly from $\mathcal M_{h}^8$ and
cross the strong double resonance along $\mathcal M_{h'}$.

\item If $h'$ is non-simple, then $\mathcal M_{h'}^8$ also becomes
a double leaf cylinder. In this case we first jump onto a simple
loop cylinder $\mathcal M_{h_i}$, cross the double resonance, and
only afterward jump onto $\mathcal M_{h'}^8$.
\end{itemize}

To summarize we expect that crumpled NHICs from \cite{BKZ}
can be continued from a $\eps^{1/6}$-neighborhood of $p_0$
to $C\sqrt \eps$-neighborhood
and can be used for diffusion. Thus,  we distinguish two
essentially different regions: ($C\sqrt \eps$-)near a strong
double resonance and ($C\sqrt \eps$-)away from it.
The main focus of this paper is the first case.

\subsection{Formulation of the main results
(small energy)}\label{formulation-main-result}
The case of finite
non-small energies is treated in  Appendix~\ref{sec:intermediate}.
We will formulate our main results in terms of the slow mechanical
system
\be \label{slow-mechanical}
H^s(I^s, \theta^s) = K(I^s) - U(\theta^s).
\ee
We make the following assumptions:
\begin{enumerate}
\item[A1.] The potential $U$ has a unique non-degenerate minimum at $0$
  and $U(0)=0$.
\item[A2.] The linearization of the Hamiltonian flow at $(0,0)$
has  distinct eigenvalues $-\lb_2<-\lb_1<0<\lb_1<\lb_2$
\end{enumerate}
In a neighborhood of $(0,0)$, there exists a local coordinate system
$(u_1, u_2, s_1, s_2)=(u,s)$ such that the $u_i-$axes correspond to
the eigendirections of $\lambda_i$ and the $s_i-$ axes correspond to
the eigendirections of $-\lambda_i$ for $i=1,2$. Let $\gamma^+$ and
$\gamma^-$ be two homoclinic orbits of $(0,0)$ under the Hamiltonian
flow of $H^s$.
This setting applies to the case of {\it a simple loop cylinder}, with
$\gamma^+=\gamma_0^{h, +}$ and $\gamma^-$ being the time reversal of
$\gamma_0^{h,+}$, denoted $\gamma_0^{h, -}$, (which is the image of
$\gamma_0^{h,+}$ under the involution $I^s\mapsto -I^s$ and $t\mapsto -t$).
We call $\gm^+$ (resp. $\gm^-$) {\it simple loop}.

We assume the following of the homoclinics $\gamma^+$ and $\gamma^-$.
\begin{enumerate}
\item[A3.] The homoclinics $\gamma^+$ and $\gamma^-$ are not tangent to
  $u_2-$axis or $s_2-$axis at $(0,0)$. This, in particular, imply that
  the curves are tangent to the $u_1$ and $s_1$ directions. We assume
  that $\gamma^+$ approaches $(0,0)$ along $s_1>0$ in the forward time,
  and approaches $(0,0)$ along $u_1 >0$ in the backward time;
  $\gamma^-$ approaches $(0,0)$ along $s_1<0$ in the forward time, and
  approaches $(0,0)$ along $u_1 <0$ in the backward time.
\end{enumerate}

For the case of the double leaf cylinder, we consider two homoclinics
$\gamma_1$ and $\gamma_2$ that are in the same direction instead of
being in the opposite direction. More precisely, the following is
assumed.

\begin{enumerate}
\item[A$3'$.]  The homoclinics $\gamma_1$ and $\gamma_2$ are not tangent
  to $u_2-$axis or $s_2-$axis at $(0,0)$. Both $\gamma_1$ and
  $\gamma_2$ approaches $(0,0)$ along $s_1>0$ in the forward time, and
  approaches $(0,0)$ along $u_1 >0$ in the backward time.
\end{enumerate}

Given $r>0$ and $0<\delta<r$, let $B_r$ be the $r-$neighborhood of
$(0,0)$ and let
$$\Sigma^s_{\pm}=\{s_1=\pm \dt\}\cap B_r, \quad  \Sigma^u_{\pm}=\{u_1=
\pm\dt\} \cap B_r$$ be four local sections contained in $B_r$. We have
four local maps
$$
\Phl^{++}:U^{++}(\subset \Sigma^s_+)\to \Sigma^u_+, \qquad
\Phl^{-+}:U^{-+}(\subset\Sigma^s_-)\to \Sigma^u_+,
$$
$$\Phl^{+-}:U^{+-}(\subset \Sigma^s_+)\to \Sigma^u_-, \qquad
\Phl^{--}:U^{--}(\subset \Sigma^s_-)\to \Sigma^u_-.
$$
The local maps are defined in the following way. Let $(u,s)$ be in the
domain of one of the local maps. If the orbit of $(u,s)$ escapes $B_r$
before reaching the destination section, then the map is considered
undefined there. Otherwise, the local map maps $(u,s)$ to the first
intersection of the orbit with the destination section. The local map
is not defined on the whole section and its domain will be made
precise later.

For the case of
simple loop cylinder, i.e. assume A3, we can define
two global maps corresponding to the homoclinics $\gamma^+$ and
$\gamma^-$. By assumption A3, for a sufficiently small $\delta$, the
homoclinic $\gamma^+$ intersects the sections $\Sigma^{u,s}_+$ and
$\gamma^-$ intersects $\Sigma^{u,s}_-$.  Let
$p^+$ and $q^+$
(resp. $p^-$ and $q^-$) be the intersection of $\gm^+$
(resp. $\gamma^-$) with $\Sigma^u_+$ and $\Sigma^s_+$
(resp. $\Sigma^u_-$ and $\Sigma^s_-$) . Smooth dependence on initial
conditions implies that for the neighborhoods $V^\pm \ni q^\pm$ there
are a well defined Poincar\'e return maps
$$
\Phg^+:V^+ \to \Sigma^s_+, \qquad \Phg^-:V^- \to \Sigma^s_-.
$$
When A$3'$ is assumed, for $i=1,2$, $\gamma^i$ intersect $\Sigma^u_+$ at
$q^i$ and intersect $\Sigma^s_+$ at $p^i$. The global maps are denoted
$$
\Phg^1:V^1 \to \Sigma^s_+, \qquad \Phg^2:V^1 \to \Sigma^s_+.
$$
\begin{figure}[t]
 \centering
  \def\svgwidth{3in}
  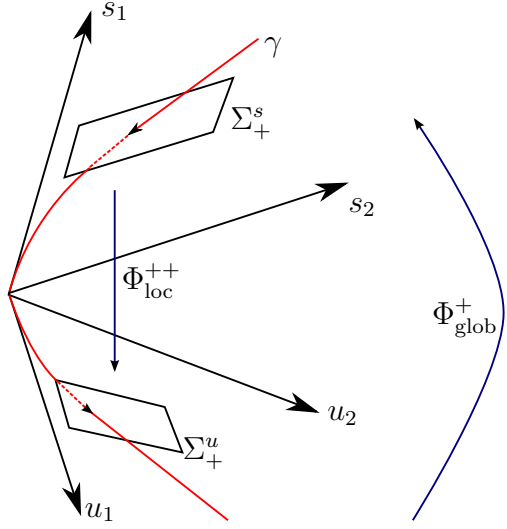
 \caption{Global and local maps for $\gamma^+$}
 \label{fig:gb-lc-maps}
\end{figure}

The composition of local and global maps for the periodic orbits
shadowing $\gamma^+$ is illustrated in Figure~\ref{fig:gb-lc-maps}.

We will assume that the global maps are ``in general position''. We
will only phrase our assumptions A4a and A4b for the homoclinic
$\gamma^+$ and $\gamma^-$. The assumptions for $\gamma^1$ and $\gamma^2$
are identical, only requiring different notations and will be called
A4a$'$ and A4b$'$. Let $W^s$ and $W^u$ denote the local stable and
unstable manifolds of $(0,0)$. Note that $W^u\cap \Sigma^u_\pm$ is
one-dimensional and contains
$q^\pm$. Let $T^{uu}(q^\pm)$ be the
tangent direction to this one dimensional curve at $q^\pm$. Similarly,
we define $T^{ss}(p^\pm)$ to be the tangent direction to $W^s\cap
\Sigma^s_\pm$ at $p^\pm$.
\begin{enumerate}
\item[A4a.] Image of strong stable and unstable directions under
  $D\Phg^\pm(q^\pm)$ is transverse to strong stable and unstable
  directions at $p^\pm$ on the energy surface $S_0=\{H^s=0\}$.
  For the restriction to $S_0$ we have
  \[
  D\Phg^+(q^+)|_{TS_0} T^{uu}(q^+) \pitchfork T^{ss}(p^+), \quad
  D\Phg^-(q^-)|_{TS_0} T^{uu}(q^-) \pitchfork T^{ss}(p^-).
  \]
\item[A4b.] Under the global map, the image of the plane
  $\{s_2=u_1=0\}$ intersects $\{s_1=u_2=0\}$ at a one dimensional
  manifold, and the intersection transversal to the strong stable and
  unstable direction. More precisely, let
  $$  L(p^\pm)= D\Phg^\pm(q^\pm)\{s_2=u_1=0\}\cap \{s_1=u_2=0\}, $$
  we have that $\dim L(p^\pm)=1$, $L(p^\pm)\ne T^{ss}
  (p^\pm)$ and $D(\Phg^\pm)^{-1}L(p^\pm)\ne T^{uu}(q^\pm)$.
\item[A$4'$.] Suppose conditions A4a and A4b hold for both $\gm_1$ and $\gm_2$.
\end{enumerate}

We show that under our assumptions, for small energy, there exists
``shadowing'' periodic orbits close to the homoclinics. These orbits
were studied by Shil'nikov \cite{Shil67}, Shil'nikov-Turaev \cite{TS89},
and Bolotin-Rabinowitz \cite{BR}.
\begin{thm}\label{periodic}
  \begin{enumerate}
  \item In the simple loop case, we assume that the assumptions A1 - A4 hold for $\gamma^+$ and $\gamma^-$. Then there exists
    $E_0>0$ such that for each $0<E\le E_0$, there exists a periodic
    orbit $\gamma_E^+$ corresponding to a fixed point of the map
    $\Phg^+\circ\Phl^{++}$ restricted to the energy surface $\cS_E$.

    For each $-E_0 \le E < 0$, there exists a periodic orbit
    $\gamma^c_E$ corresponding to a fixed point of the map $\Phg^-
    \circ \Phl^{+-} \circ \Phg^+ \circ \Phl^{-+}$ restricted to the
    energy surface $\cS_E$.

    For each $0 < E \le E_0$, there exists a periodic orbit
    $\gamma_E^-$ corresponding to a fixed point of the map $\Phg^-
    \circ \Phl^{--}$ restricted to the energy surface $\cS_E$.
  \item In the non-simple case, assume that the assumptions A1, A2, A3$\,'$ and A4$'$ hold for $\gamma^1$ and $\gamma^2$. Then
    there exists $E_0 >0$ such that for $0<E\le E_0$, the following hold. For any $\sigma=(\sigma_1, \cdots, \sigma_n)$,  there is a
    periodic orbit $\gamma^\sigma_E$ corresponding to a fixed point of
    the map 
    $$ \prod_{i=n}^1 \left(\Phg^{\sigma_i} \circ  \Phl^{++} \right)$$
    restricted to the energy surface $\cS_E$. (Product stands for composition of maps). 
  \end{enumerate}
\end{thm}

\begin{figure}[t]
  \centering
  \includegraphics[width=5in]{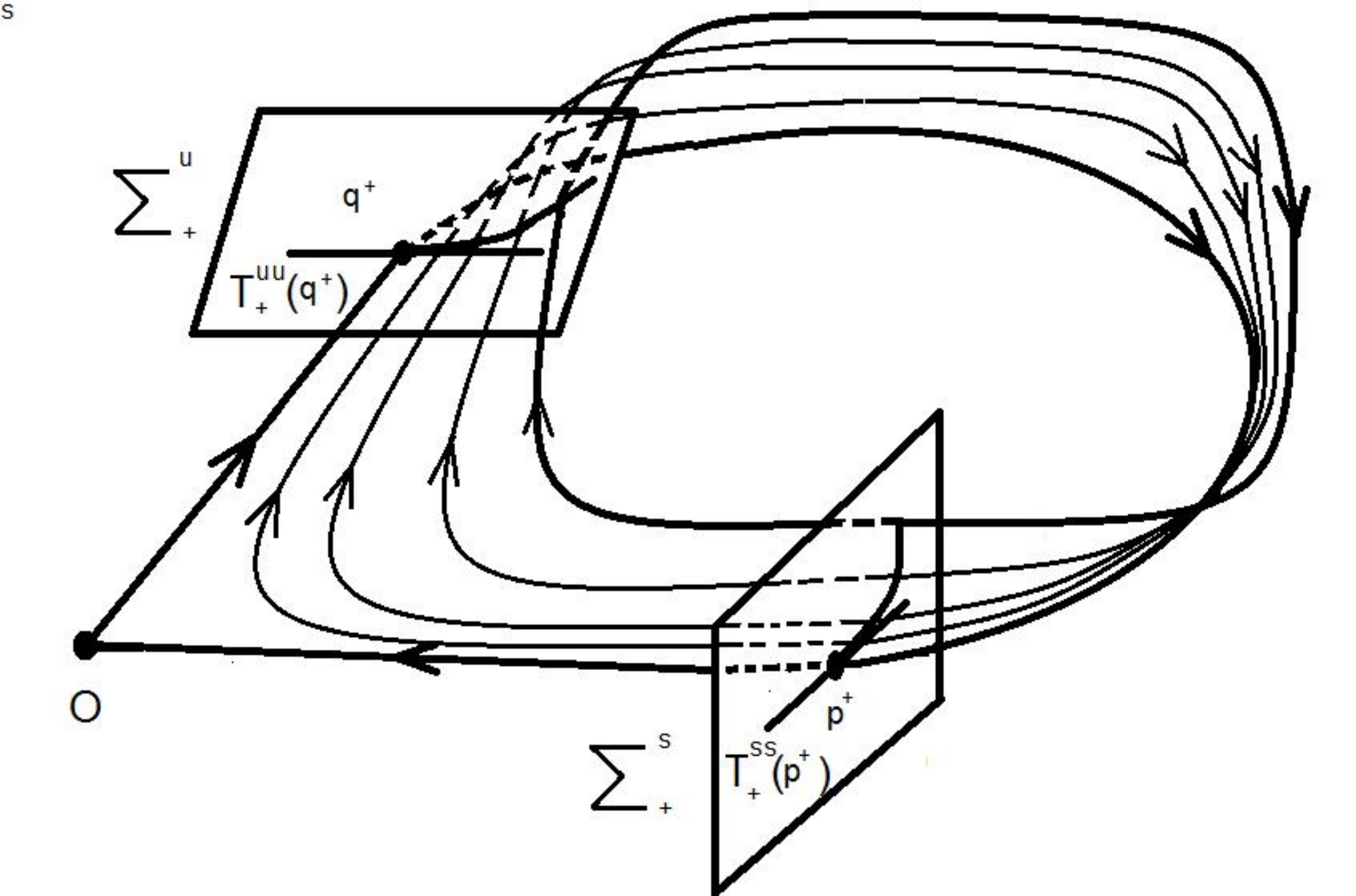}
  \caption{Periodic orbits shadowing $\gamma^+$}
  \label{fig:return-map}
\end{figure}

The the periodic orbits $\gamma^+_E$ are depicted in Figure~\ref{fig:return-map}.

\begin{thm} \label{NHIC-mechanical} In the case of simple loop, assume
  that  A1-A4 are satisfied with $\gamma^+=\gamma^{h,+}_0$ and
  $\gamma^-=\gamma^{h,-}_0$. For this choice of $\gamma^+$ and
  $\gamma^-$, let $\gamma^+_E$, $\gamma^c_E$ and $\gamma^-_E$ be the
  periodic orbits obtained from part 1 of Theorem~\ref{periodic} .
    \[
  \sM_{h}^{E_0}=\bigcup_{0<E\le E_0} \gamma^+_E \cup \gamma \cup
  \bigcup_{-E_0 \le E < 0}\gamma^c_E \cup \gamma^- \cup \bigcup_{0<E
    \le E_0} \gamma^-_E
  \]
  is a $C^1$ smooth normally hyperbolic invariant manifold with
  boundaries $\gamma^+_{E_0}$, $\gamma^c_{E_0}$ and $\gamma^-_{E_0}$.

In the case of non-simple loop, assume that A1, A2, A3\,$'$ and A4$'$ are
satisfied with $\gamma^1 = \gamma^{h_1}_0$ and $\gamma^2 =
\gamma^{h_2}_0$.  
Let $\gamma^\sigma_E$ denote the periodic
orbits obtained from applying part 2 of Theorem~\ref{periodic} to the sequence $\sigma$ determined by Lemma~\ref{homotopy}. We
have that for any $e>0$, the set
\[
\sM_{h}^{e,E_0}=\cup_{e \le E \le E_0} \gm^\sigma_E
\]
is a $C^1$ smooth normally hyperbolic invariant manifold.
\end{thm}

\begin{rmk} \label{kissing-cylinders} Due to hyperbolicity the
  cylinder $\sM_h^{E_0}$ is $C^\al$ for any $0<\al<\lb_2/\lb_1$.

If $h_1$ and $h_2$ corresponds to  simple loops, then the
corresponding invariant manifolds $\sM_{h_1}^{E_0}$ and
$\sM_{h_2}^{E_0}$ have a tangency along a two
dimensional plane at the origin. One can say that we have
``kissing manifolds'', see Figure \ref{fig:kissing}.
\end{rmk}

\begin{rmk}
  In the simple loop case, we expect the shadowing orbits $\gamma^\pm_E$, for $0 \le E \le E_0$ to coincide with the minimal geodesics $\gamma^{\pm h}_E$. In the non-simple case, $\gamma^\sigma_E$ should coincide with $\gamma_E^h$ for $0 \le E \le E_0$ (by Lemma~\ref{homotopy}, $\sigma$ is uniquely determined by $h$). The proof is not included in this paper, as we only deal with the geometrical part of the diffusion. 
\end{rmk}

\begin{cor} \label{existence-NHIC}
The system $H_\eps$ has a normally hyperbolic
manifold $\sM_{h,\eps}^{E_0}$ (resp. $\sM_{h,\eps}^{e,E_0}$)
which is weakly invariant, i.e. the Hamiltonian vector field
of $H_\eps$ is tangent to $\sM_{h,\eps}^{E_0}$
(resp. $\sM_{h,\eps}^{e,E_0}$). Moreover, the intersection of
$\sM_{h,\eps}^{E_0}$ (resp. $\sM_{h,\eps}^{e,E_0}$) with
the regions  $\{-E_0\le H^s \le E_0\}\times \T$
(resp. $\{e\le H^s \le E_0\}\times \T$) is a
$C^1$-graph over
$\sM_{h}^{E_0}$ (resp. $\sM_{h}^{e,E_0}$).
\end{cor}

Proof of Corollary \ref{existence-NHIC} is included in section~\ref{sec:gen-nhic}.




\section{Normal form near the hyperbolic fixed point}
\label{sec:system-normal-form}

In a neighborhood of the origin, there exists a a symplectic linear
change of coordinates under which the system has the normal form
$$ H(u_1, u_2, s_1, s_2) = \lambda_1 s_1 u_1 + \lambda_2 s_2 u_2 + O_3(s,u).$$
Here $s=(s_1,s_2)$, $u=(u_1, u_2)$, and $O_n(s,u)$ stands for a
function bounded by $C|(s,u)|^n$. According to our assumptions,
$\lambda_1 < \lambda_2$.

The main result of this section is the following normal form
\begin{thm}\label{normal-form}
  There exists $k\in \NN$ depending only on $\lambda_2/\lambda_1$ such
  that if $H$ is $C^{k+1}$, the following hold. There exists
  neighborhood $U$ of the origin and a $C^2$ change of coordinates $\Phi$
  on $U$ such that $N_k=H \circ \Phi$
  has the form
  is a polynomial of degree $k$ of the form
    \begin{equation}\label{eq:norm-form2}
    \begin{bmatrix}
      \dot{s}_1 \\ \dot{s}_2 \\ \dot{u}_1 \\ \dot{u}_2
    \end{bmatrix} =
    \begin{bmatrix}
      -\partial_{u_1}N_k \\ -\partial_{u_2}N_k \\ \partial_{s_1}N_k
      \\ \partial_{s_2}N_k
    \end{bmatrix}
    =
    \begin{bmatrix}
      -\lambda_1 s_1 + F_1(s,u)  \\
      -\lambda_2 s_2 + F_2(s,u)  \\
     \ \ \lambda u_1 + G_1(s,u) \\
     \ \ \lambda u_2 + G_2(s,u)
    \end{bmatrix}
  \end{equation}
  where
$$ F_1 = s_1O_1(s,u) + s_2O_1(s,u), \quad  F_2 = s_1^2O(1) + s_2O_1(s,u), $$
$$ G_1 =  u_1O_1(s,u) + u_2O_1(s,u), \quad G_2 =  u_1^2O(1) + u_2 O_1(s,u) .$$
\end{thm}

The proof consists of two steps: first, we do some preliminary normal
form and then apply a theorem of Belitskii-Samovol (See, for example \cite{IL99}).

Since $(0,0)$ is a hyperbolic fixed point, for sufficiently small
$r>0$, there exists stable manifold $W^s=\{(u=U(s), |s|\le r\}$ and
unstable manifold $W^u=\{s=S(u), |u|\le r\}$ containing the
origin. All points on $W^s$ converges to $(0,0)$ exponentially in
forward time, while all points on $W^u$ converges to $(0,0)$
exponentially in backward time. These manifolds are Lagrangian; as a
consequence, the change of coordinates $s'=s-S(u)$,
$u'=u-U(s')=u-U(s-S(u))$ is symplectic. Under the new coordinates, we
have that $W^s=\{u'=0\}$ and $W^u=\{s'=0\}$. We abuse notation and
keep using $(s,u)$ to denote the new coordinate system.

Under the new coordinate system, the Hamiltonian has the form
$$ H(s,u) = \lambda_1 s_1 u_1 + \lambda_2 s_2 u_2 + H_1(s,u), $$
where $H(s,u) = O_3(s,u)$ and $H_1(s,u)|_{s=0} =
H_1(s,u)|_{u=0}=0$. Let us denote $H_0=\lambda_1 s_1 u_1 + \lambda_2
s_2 u_2$. We now perform a further step of normalization.

We say an tuple $(\alpha, \beta) \in \NN^2 \times \NN^2$ is
{\it resonant } if $ \sum_{i=1}^2 \lambda_i (\alpha_i-\beta_i) =0.$
Note that an $(\alpha, \beta)$ with $\alpha_i=\beta_i$ for $i=1,2$ is
always resonant. A monomial $u_1^{\alpha_1} u_2 ^{\alpha_2}
s_1^{\beta_1} s_2^{\beta_2}$ is resonant if $(\alpha,\beta)$ is
resonant. Otherwise, we call it {\it nonresonant}.  It is well
known that a Hamiltonian can always be transformed, via a formal
power series, to an Hamiltonian with only resonant terms.

\begin{prop}
  If $H$ is at least $C^{k+1}$, the there exists a
  $C^\infty-$symplectic change of coordinates $(s,u) = \Phi(s', u')$
  defined on a neighborhood of $(0,0)$ such that
  $$ H \circ \Phi' = N_k(s',u') + H_2(s',u'),$$
  where $N_k$ is a polynomial of degree $k$ consisting only of
  resonant terms and $H_2 = O_{k+1}(s',u')$.
\end{prop}
\begin{proof}
  Let $S_k$ denote the set of all nonresonant indices $(\alpha, \beta)\in
  \NN^2\times \NN^2$ with $|\alpha| + |\beta|=k$. We define the change
  of coordinates by the generating function
  $$ G_k(s,u') = s_1 u_1' + s_2 u_2' + \sum_{3 \le i \le k+1}
  \sum_{(\alpha, \beta)\in S_i}g_{\alpha, \beta}s^\alpha(u')^\beta.$$
  The symplectic change of coordinates is defined by
  $s'=\partial_{u'}G_k$ and $u=\partial_{s}G_k$. Assume that
  $$ H\circ \Phi = \sum_{i \ge 2} \sum_{|\alpha|+|\beta|}h_{\alpha, \beta}(s')^\alpha (u')^\beta. $$
  We have that if $(\alpha, \beta)$ is nonresonant, there exists a
  unique $g_{\alpha, \beta}$ such that $h_{\alpha, \beta}=0$ (see
  \cite{SM95}, section 30, for example). By choosing $g_{\alpha,
    \beta}$ appropriately, we obtain the desired normal form.
 \end{proof}

We abuse notations by replacing $(s',u')$ with $(s,u)$. Using our assumption that
$0 < \lambda_1 < \lambda_2$, we have that all $(\alpha, \beta)$ with
$\alpha \ne \bt$,  $\alpha_1=1$ and $\alpha_2=0$ are nonresonant,
and similarly, all $(\alpha, \beta)$ with $\alpha \ne \bt$,
$\beta_1=1$ and $\beta_2=0$ are nonresonant. Furthermore, by performing
the straightening of stable/unstable manifolds again if necessary,
we may assume that $N_k|_{s=0} = N_k|_{u=0} =0$. As a consequence,
the normal form $N_k$ must take the following form:

\begin{cor} \label{normal-hamiltonian-form}
The normal form $N_k$ satisfies
  $$ N_k = \lambda_1 s_1u_1 + \lambda_2 s_2u_2 + s_2O_1(u)O_1(s,u) +  s_1^2O_1(u) + u_2O_1(s)O_1(s,u) + u_1^2O_1(s)$$
In particular, we have
$N_k = \lambda_1 s_1u_1 + \lambda_2 s_2u_2 + O_3(s,u)$.
\end{cor}




Under the normal form the equations of motion is
\begin{equation}\label{eq:norm-form1}
  \begin{cases}
    \dot{s} = -\partial_uN_k + O_k(s,u) \\
    \dot{u} = \partial_sN_k + O_k(s,u)
  \end{cases}.
\end{equation}
As the linearization of these equations is hyperbolic, for
sufficiently large $r$ it is possible to kill the small remainder with
a finitely smooth change of coordinates.

Theorem \ref{normal-form} is a direct consequence of the following theorem:
\begin{thm}[Belitskii-Samovol](See \cite{IL99}, Chapter 6, Theorem
  1.6) For any $l\in \NN$ and $\lambda\in \C^n$ with $Re \lambda_i \ne
  0$, there exists an integer $k=k(l,\lambda)$ such that the following
  hold. Suppose two germs of vector fields at a hyperbolic fixed point
  with the spectrum of linearization equal to $\lambda$, and their
  jets of order $k$ coincide at the fixed point. Then the two vector
  fields are $C^l-$conjugate.
\end{thm}

\section{Behavior of a family of orbits passing near $0$
and Shil'nikov boundary value problem}
The main result of this section is the following
\begin{thm}\label{tangency}
  Let $(s^T,u^T)$ be a family of orbits satisfying $s^T(0)\to s^{in}$
  as $T\to \infty$
  with $s_1^T=\dt$ and $u^T(T)\to u^{out}$ as $T\to \infty$
  with $s_1^T=\dt$ with $|s^T|, |u^T|\le 2 \delta$, where $\delta$
  is small enough. Then there exists $T_0,\ C>0$ and $\alpha>1$
  such that for each $T>T_0$ and all $0\le t \le T$ we have
  $$
  |s_2^T(t)|\le C |s_1^T(t)|^\alpha ,
  \quad |u_2^T(t)| \le C   |u_1^T(t)|^\alpha.
  $$
  In particular, the curve $\{(s_1^T(T), s_2^T(T))\}_{T\ge T_0}
  \Sigma^u_+=\{s_1^T(0)=\dt\}$ is tangent to the $s_1$--axis
  at $T=\infty$ and
  $\{(u_1^T(0), u_2^T(0))\}\subset \Sigma^s_+=\{s_1^T(0)=\dt\}$
  is tangent to the $u_1$--axis at $T=\infty$.
\end{thm}

We will use the local normal form to study the local maps. Our main
technical tool to prove the above Theorem is the following
{\it boundary value problem due to  Shil'nikov} (see \cite{Shil67}):
\begin{prop}\label{bvp}
  There exists $\epsilon_0>0$ such that for any $0<\epsilon\le
  \epsilon_0$, there exist $\delta>0$ such that the following
  hold. For any $s^{in}=(s_1^{in},s_2^{in})$, $u^{out}=(u_1^{out},
  u_2^{out})$ with $|s|, |u|\le \delta$ and any large $T>0$, there
  exists a unique solution $(s^T,u^T):[0,T]\to B_\delta$ of the system
  (\ref{eq:norm-form2}) with the property $s^T(0)=s^{in}$ and
  $u^T(T)=u^{out}$. Let
  \begin{equation}\label{eq:su1}(s^{(1)},u^{(1)})(t)=(e^{-\lambda_1
      t}s_1^{in}, e^{-\lambda_2 t}s_2^{in},
    e^{-\lambda_1(T-t)}u_1^{out} , e^{-\lambda_2(T-t)}u_2^{out}),
  \end{equation}
  we have
$$ |s_{1}^T(t)-s_1^{(1)}(t)|\le  \delta  e^{-(\lambda_1-\epsilon) t},  \quad |s_{2}^T(t)-s_2^{(1)}(t)|\le \delta e^{-(\lambda_2'-2\epsilon) t}, $$
$$ |u_{1}^T(t)-u_1^{(1)}(t)|\le \delta e^{-(\lambda_1-\epsilon)(T-t)},   \quad |u_{2}^T(t)-u_2^{(1)}(t)|\le \delta e^{-(\lambda_2'-2\epsilon)(T-t)},$$ where $\lambda_2'=\min\{\lambda_2, 2\lambda_1\}$. Furthermore, for $s_1$ and $u_1$, we have an additional lower bound estimate:
\begin{equation}\label{eq:lowerbd} |s_1^T(t)|\ge \frac 12\ |s_1^{in}|\
  e^{-(\lambda_1+\epsilon)t}, \quad |u_1^T(t)|\ge \frac 12\ |u_1^{out}|\
  e^{-(\lambda_1 +\epsilon)(T-t)}.
\end{equation}
Note that for (\ref{eq:lowerbd}) to hold, the choice of $\delta$ needs
to depend on a lower bound for $|s_1^{in}|$ and $|u_1^{out}|$.
\end{prop}

\begin{proof}
  Let $\Gamma$ denote the set of all smooth curves $(s,u):[0,T]\to
  B(0,\delta)$ such that the $s(0)=(s_1^{in},s_2^{in})$ and
  $u(T)=(u_1^{out}, u_2^{out})$. We define a map $\mF:\Gamma \to
  \Gamma$ by $\mF(s,u) = (\tilde{s}, \tilde{u})$, where
  $$ \begin{aligned}
    \tilde{s}_1 & = e^{-\lambda_1 t}s_1^{in} + \int_0^te^{\lambda_1(\xi -t)}   F_1(s(\xi),u(\xi))d\xi ,\\
    \tilde{s}_2 & = e^{-\lambda_2 t}s_2^{in} + \int_0^te^{\lambda_2(\xi -t)}   F_2(s(\xi),u(\xi))d\xi ,\\
    \tilde{u}_1 & = e^{-\lambda_1(T-t)}u_1^{out} - \int_t^T e^{-\lambda_1     (\xi-t)}G_1 (s(\xi), u(\xi)) d\xi, \\
    \tilde{u}_2 & = e^{-\lambda_2(T-t)}u_2^{out} - \int_t^T
    e^{-\lambda_2 (\xi-t)}G_2 (s(\xi), u(\xi)) d\xi .  \end{aligned} $$

It is proved in \cite{Shil67} that for sufficiently small $\delta$,
the map $\mF$ is a contraction in the uniform norm. Let $s^{(1)}, u^{(1)}$
be as defined in (\ref{eq:su1}) and $(s^{(k+1)},
u^{(k+1)})=\mF(s^{(k)}, u^{(k)})$, then $(s^{(k)}, u^{(k)})$ converges
to the solution of the boundary value problem. Using the normal form
(\ref{eq:norm-form2}), we will provide precise estimates on the
sequence $(s^{(k)},u^{(k)})$. The upper bound estimates are
consequences of the following:
\begin{align*} &|s_{1}^{(k+1)}(t) - s_{1}^{(k)}(t)|\le 2^{-k} \delta   e^{-(\lambda_1-\epsilon) t}, \qquad \quad |s_{2}^{(k+1)}(t) - s_{2}^{(k)}(t)|\le   2^{-k}\delta e^{-(\lambda_2'-\epsilon) t}, \\
  &|u_{1}^{(k+1)}(t) - u_{1}^{(k)}(t)|\le 2^{-k} \delta
  e^{-(\lambda_1-\epsilon)(T-t)}, \quad |u_{2}^{(k+1)}(t) -
  u_{2}^{(k)}(t)|\le 2^{-k} \delta e^{-(\lambda_2'-\epsilon)(T-t)}.
\end{align*}

We have
\begin{align*}
  &|s_1^{(2)}(t)-s_1^{(1)}(t)| = \int_0^t e^{\lambda_1(\xi -t)}\left|     s_1^{(1)}(\xi)\, O_1(s,u) + s_2^{(1)} 
  (\xi)\,O_1(s,u)\right|d\xi \\
  & \le \int_0^T e^{\lambda_1(\xi -t)} (O(\delta^2) e^{-\lambda_1 \xi} +   O(\delta^2) e^{-\lambda_2\xi}) d\xi \\
  & \le O(\delta^2) t e^{-\lambda_1 t}
  \le C \dfrac{te^{\eps \, t}}{\epsilon t}\delta^2
  e^{-(\lambda_1-\epsilon)t} \le C \epsilon^{-1}\delta^2
  e^{-(\lambda_1-\epsilon)t} \le
  \frac12\delta e^{-(\lambda_1-\epsilon)t}.
\end{align*}
Note that the last inequality can be guaranteed by choosing $\delta
\le C^{-1}\epsilon$.  Similarly
\begin{align*}
  &|s_2^{(2)}(t)-s_2^{(1)}(t)| = \int_0^t e^{\lambda_2(\xi -t)}\left|     
  (s_1^{(1)}(\xi))^2 O(1) + s_2^{(1)}(\xi)O_1(s,u)\right|d\xi \\
  & \le \int_0^t e^{\lambda_2(\xi -t)} (
  O(\delta^2) e^{-2\lambda_1 \xi} +   O(\delta^2) e^{-\lambda_2\xi}) d\xi \\
  & \le O(\delta^2) \int_0^t e^{\lambda_2'(\xi -t)}e^{-\lambda_2' \xi}
  d\xi
   \le C \delta^2\,t e^{-\lambda_2' t}
  \le C \delta^2\,\dfrac{e^{2\eps \, t}}{2\epsilon}
  e^{-\lambda_2't} \\
  & \le C \epsilon^{-1}\delta^2 e^{-(\lambda_2'-\epsilon)t} \le
  \frac12\delta e^{-(\lambda_2'-2\epsilon)t}.
\end{align*}
Observe that the calculations for $u_1$ and $u_2$ are identical if we
replace $t$ with $T-t$. We obtain
$$ |u^{(2)}_1(t)-u^{(1)}_1(t)| \le \frac12\delta e^{-(\lambda_1-\epsilon)(T-t)}, \quad |u^{(2)}_2(t)-u^{(1)}_2(t)| \le \frac12\delta e^{-(\lambda_2'-2\epsilon)(T-t)}. $$

According to the normal form~(\ref{eq:norm-form2}), we have there
exists $C'>0$ such that
$$ \|\partial_{s}F_1\| \le C' \|(s,u)\| , \quad \|\partial_{u}F_1\| \le C' \|s\|. $$ Using the inductive hypothesis for step $k$, we have $\|s^{(k)}(t)\| \le 2\delta e^{-(\lambda_1-\epsilon)t}$. It follows that
\begin{align*}
  &|s^{(k+2)}_1(t)-s^{(k+1)}_1(t)|\\
  &\le \int_0^t e^{\lambda_1(\xi-t)}\left( \|\partial_{s}F_1\|\,     \|s^{(k+1)}-s^{(k)}\| + \|\partial_{u}F_1\|\, \|u^{(k+1)}-u^{(k)}\| \right)   d\xi   \\
  & \le C' \int_0^t e^{\lambda_1(\xi-t)} \left( \delta 2^{-k}\delta     e^{-(\lambda_1-\epsilon)\xi} + \delta     e^{-(\lambda_1-\epsilon)\xi}2^{-k}\delta    \right) d\xi \\
  & \le 2^{-k}\delta e^{-(\lambda_1-\epsilon)t} \int_0^t 2C'
  e^{-\epsilon \xi}\delta d\xi \le 2^{-(k+1)} \delta
  e^{-(\lambda_1-\epsilon)t}.
\end{align*}
Note that the last inequality can be guaranteed by choosing $\delta$
sufficiently small depending on $C'$ and $\epsilon$. The estimates for
$s_2$ needs more detailed analysis. We write
\begin{align*}
  & |s^{(k+2)}_2(t) - s^{(k+1)}_2(t)| \le \int_0^t e^{\lambda_2(\xi - t)}\cdot   \\
  &\left( \|\partial_{s_1}F_2\| |s_1^{(k+1)}-s_1^{(k)}| + \|\partial_{s_2}F_2\|     |s_2^{(k+1)}-s_2^{(k)}| + \|\partial_{u}F_2\|     \|u_2^{(k+1)}-u_2^{(k+1)}\| \right) d\xi \\
  & = \int_0^te^{\lambda_2(\xi-t)}(I + II + III) d\xi.
\end{align*}
We have $\|\partial_{s_1}F_2\| = O_1(s_1)O(1) + O_1(s_2)O(1)$,
hence
$$ I \le C' (\delta e^{-(\lambda_1-\epsilon)\xi} + \delta e^{-(\lambda_2'-2\epsilon)\xi}) 2^{-k} \delta e^{-(\lambda_1-\epsilon)\xi} \le C' 2^{-k}\delta2 e^{-2(\lambda_1'-\epsilon)\xi}.$$
Since $\|\partial_{s_2}F_2\| = O_2(s_1) + O_1(s,u)=O_1(s,u)$, we have $ II \le C' \delta^2 2^{-k} e^{-(\lambda_2'-2\epsilon)\xi}$. Finally, as $\|\partial_uF_2\| = O_2(s_1) + O_1(s_2)O(1)$, we have
$$ III \le C' 2^{-k} \delta (\delta^2 e^{-2(\lambda_1-\epsilon)\xi} + \delta e^{-(\lambda_2'-2\epsilon)\xi}) \le C' 2^{-k}\delta^2 e^{-(\lambda_2'-2\epsilon)\xi}.$$ Note that in the last line, we used $\lambda_2' \le 2\lambda_1 $. Combine the estimates obtained, we have
\begin{multline*} |s^{(k+2)}_2(t) - s^{(k+1)}_2(t)| \le \delta 2^{-k} \int_0^t   3C' \delta e^{\lambda_2(\xi - t)}e^{-(\lambda_2'-2\epsilon)\xi} d\xi \\
  \le \delta 2^{-k}e^{-(\lambda_2'-2\epsilon)t} \int_0^t 3C'
  \delta e^{-2\epsilon \xi}d\xi \le 2^{-(k+1)}\delta
  e^{-(\lambda_2'-2\epsilon)t}.
\end{multline*}
The estimates for $u_1$ and $u_2$ follow from symmetry.

We now prove the lower bound estimates (\ref{eq:lowerbd}). We will
first prove the estimates for $s_1$ in the case of $s_1^{in}>0$.  We
have the following differential inequality
$$ \dot{s}_1 \ge  -(\lambda_1 + C'\delta)s_1 + s_2 O_1(s,u).$$
Note that $|s_2(t)|\le 2 \delta e^{-\lambda_2' t}$ due to the already
established upper bound estimates. Choose $\delta$ such that
$C'\delta \le \epsilon$, we have
\begin{align*} s_1(t) & \ge s_1^{in}e^{-(\lambda_1 +\epsilon)t} - \int_0^t   e^{-(\lambda_1+\epsilon)(\xi-t)}2\delta e^{-(\lambda_2'-2\epsilon) \xi}\cdot   C'   \delta d\xi \\
  & \ge s_1^{in}e^{-(\lambda_1 +\epsilon)t} - 2C'\delta^2
  (\lambda_2'-\lambda_1-3\epsilon)^{-1} e^{-(\lambda_1+\epsilon)t} \ge
  \frac12 s_1^{in}e^{-(\lambda_1 +\epsilon)t}.
\end{align*}
For the last inequality to hold, we choose $\epsilon_0$ small enough
such that $\lambda_2'-\lambda_1-3\epsilon>0$, and choose $\delta$ such
that $2C'\delta^2(\lambda_2'-\lambda_1-3\epsilon)^{-1}\le \frac12
s_1^{in}$.

The case when $s_1^{in}<0$ follows from applying the above analysis to
$-s_1$. The estimates for $u_1$ can be obtained by replacing $s_i$
with $u_i$ and $t$ with $T-t$ in the above analysis.
\end{proof}

\begin{proof}[Proof of Theorem \ref{tangency}]

  It follows from Proposition~\ref{bvp} that $|s_1^T(t)|\ge
  \frac12|s_1^{in}| e^{-(\lambda_1+\epsilon)t}$ and $|s_2^T(t)|\le
  2\delta e^{-(\lambda_2'-2\epsilon)t}$. We obtain the estimates for
  $s_1$ and $s_2$ by choosing $\alpha =
  \frac{\lambda_2-2\epsilon}{\lambda_1+\epsilon}$ and
  $C=4\delta/|s_1^{in}|$. The case of $u_1$ and $u_2$ can be proved
  similarly. \end{proof}

\section{Properties of the local maps}
\label{sec:local}

 Denote $p^\pm = (s^\pm , 0) = \gamma^\pm  \cap \Sigma^s_\pm$ and $q^\pm = (0, u^\pm) =\gamma^\pm \cap \Sigma^u_\pm$. Although the local map $\Phl^{++}$ is not defined at $p^+$ (and its inverse is not defined at $q^+$), the map is well defined from a neighborhood close to $p^+$ to a neighborhood close to $q^+$. In particular, for any $T>0$, by Proposition~\ref{bvp}, there exists a trajectory $(s,u)^{++}_T$ of the Hamiltonian flow such that
$$ s_T^{++}(0) = s^+, \quad u_T^{++}(T) = u^+. $$
Denote $x^{++}_T = (s,u)^{++}_T(0)$ and $y^{++}_T = (s,u)^{++}_T(T)$, we have $\Phl^{++}(x^{++}_T) = y^{++}_T$, and $x^{++}_T \to p^+$, $y^{++}_T \to q^+$ as $T\to \infty$. We apply the same procedure to other local maps and extend the notations by changing the
superscripts accordingly.

Let $N=N_k(s,u)$ be the Hamiltonian from Theorem \ref{normal-form},
$E(T)=N((s,u)^{++}_T)$ be the energy of the orbit, and $S_{E(T)}=\{N=E(T)\}$
be the corresponding energy surface.
We will show that the domain of $\Phl^{++}|_{S_{E(T)}}$
can be extended to a larger
subset of $\Sigma^{s, E(T)}_+$ containing $x_T^{++}$.   We call
$R \subset \Sigma^s_+\cap S_{E(T)}$ a rectangle if it is bounded by four vertices
$x_1, \cdots, x_4$ and $C^1$ curves $\gamma_{ij}$ connecting $x_i$ and $x_j$,
where $ij\in \{12, 34, 13, 24\}$. The curves does not intersect except at
the vertices. Denote $B_\dt(x)$ the $\dt$-ball around $x$ and
the local parts of
invariant manifolds
$$
T_s^+=W^s(0)\cap \Sigma^s_{+} \cap B_\delta(p^+), \quad
T_u^+=W^u(0)\cap
\Sigma^u_{+} \cap B_\delta(q^+)
$$
and the $\Sigma$-sections restricted to an energy surface $S_E$ by
$$
\Sigma^{s,E}_+=\Sigma^{s}_+ \cap S_E \quad \text{  and }
\quad \Sigma^{u,E}_+=\Sigma^{u}_+ \cap S_E.
$$

\begin{figure}[t]
 \centering
  \def\svgwidth{5in}
 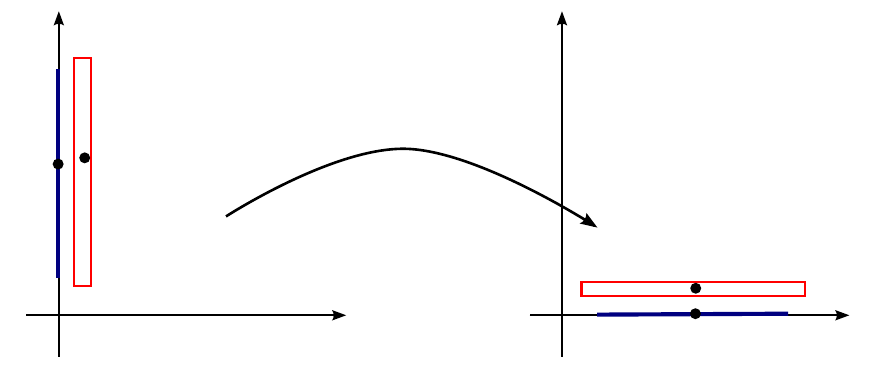
 \caption{Local map $\Phl^{++}$}
 \label{fig:localmap}
\end{figure}


The main result of this section is the following
\begin{thm}\label{local-domain}
  There exists $\delta_0>0$ and $T_0>0$ such that for any $T>T_0$ and $0< \delta < \delta_0$,
  there exists a rectangle $R^{++}(T)\subset \Sigma^{s,E(T)}_+$, with
  vertices $x_i(T)$ and
  $C^1$-smooth sides $\gamma_{ij}(T)$, such that the following hold:
  \begin{enumerate}
  \item $\Phl^{++}$ is well defined on $R^{++}(T)$. $\Phl^{++}(R^{++}(T))$ is also
  a rectangle  with  vertices $x_i'(T)$ and  sides   $\gamma'_{ij}(T)$.
  \item As $T\to 0$, $\gamma_{12}(T)$ and $\gamma_{34}(T)$ both converge in Hausdorff
  metric to a single curve containing $T_s^+$;
  $\gamma'_{13}(T)$ and $\gamma'_{24}(T)$ converges to a single curve
  containing  $T_u^+$.
  \end{enumerate}
  The same conclusions, after substituting the superscripts according to
  the signatures of the map, hold for other local maps.
\end{thm}

To get a picture of Theorem~\ref{local-domain}, note that for a given energy
$E > 0$,
the restricted sections $\Sigma^{s,E}_+$ and $\Sigma^{u,E}_+$ are both
transversal
to the $s_1$ and $u_1$ axes, and hence these sections can be parametrized
by the $s_2$ and $u_2$ components. An illustration of the local maps and
the rectangles is contained in Figure~\ref{fig:localmap}.

 We will only prove Theorem~\ref{local-domain} for the local map $\Phl^{++}$.
 The proof for the other local maps are identical with proper changes of notations.


Let $(v_{s_1}, v_{s_2}, v_{u_1}, v_{u_2})$ denote the coordinates for
the tangent space induced by $(s_1, s_2, u_1, u_2)$.
As before $B_r$ denotes the $r-$neighborhood of
the origin.
For $c>0$ and $x\in B_r$, we define the
\emph{strong unstable cone} by
$$ C^{u,\,c}(x)=\{c|v_{u_2}|^2 > |v_{u_1}|^2 + |v_{s_1}|^2 + |v_{s_2}|^2\}$$
and the \emph{strong stable cone} to be
$$ C^{s,\,c}(x)=\{c|v_{s_2}|^2 > |v_{s_1}|^2 + |v_{u_1}|^2 + |v_{u_2}|^2\}.$$
The following properties follows from the fact that the linearization
of the flow at $0$ is hyperbolic. We will drop the superscript $c$ when
the dependence in $c$ is not stressed.

\begin{lem}\label{local-cone}
For any
$0<\epsilon<\lambda_2-\lambda_1$, there exists $r=r(\epsilon, c)$ such
that the following holds:
\begin{itemize}
\item If $\phi_t(x)\in B_r$ for $0\le t \le t_0$, then $
  D\phi_t(C^{u}(x))\subset C^{u}(\phi_t(x))$ for all $0\le t\le
  t_0$. Furthermore, for any $v\in C^{u}(x)$,
$$ |D\phi_t(x)v| \ge e^{(\lambda_2-\epsilon)t}, \quad 0 \le t\le t_0.$$
\item If $\phi_{-t}(x)\in B_r$ for $0\le t \le t_0$, then $
  D\phi_{-t}(C^{s}(x))\subset C^{s}(\phi_{-t}(x))$ for all $0\le
  t\le t_0$. Furthermore, for any $v\in C^{s}(x)$,
$$ |D\phi_{-t}(x)v| \ge e^{(\lambda_2-\epsilon)t}, \quad 0 \le t\le t_0.$$
\end{itemize}
\end{lem}

For each energy surface $E$, we define the restricted cones
$C^{u}_E(x)=C^{u}(x)\cap T_xS_E$ and $C^{s}_E(x)=C^{s}(x)\cap
T_xS_E$.

{\bf Warning:} Recall that the Hamiltonian $N$ under consideration
by Theorem \ref{normal-form}
has the form $N_k = \lambda_1 s_1u_1 + \lambda_2 s_2u_2 + O_3(s,u)$.
It is easy to see that the restricted cones $C^{u}_E(x)$ and
$C^{s}_E(x)$ might {\it be empty}. Excluding this
case requires  a special care!

Since the energy surface is invariant under the flow, its
tangent space is also invariant.
We have the following observation:
\begin{lem}
  If $\phi_t(x)\in B_r$ for $0\le t \le t_0$, then $C^{u}_E$ is
  invariant under the map $D\varphi_t$ for $0\le t\le t_0$. In
  particular, if $C^{u}_E(x)\ne \emptyset$, then
  $C^{u}_E(\phi_t(x))\ne \emptyset$. Similar conclusions hold for
  $C^{s}_E$ with $\phi_{-t}$.
\end{lem}

Let $x$ be such that $\phi_t(x)\in B_r\cap S_E$ for $0\le t\le
t_0$. A Lipschitz curve $\gamma_E^s(x)$ is called
\emph{stable } if its
forward image stays in $B_r$ for $0\le t\le t_0$, and that the
curve and all its forward images are tangent to the restricted stable
cone field $\{C^{s}_E\}$. For $y$ such that
$\phi_{-t}(y)\in B_r\cap S_E$ for $0\le t\le t_0$, we may define
the \emph{unstable curve} $\gamma_E^u(y)$ in the same way with $t$
replaced by $-t$ and $C^{s}_E$ replaced by $C^{u}_E$.
Notice that stable and unstable curves are {\it not }
in the tangent space, but in the phase space.

\begin{prop}\label{stable-curve}
  In notations of Lemma \ref{local-cone}
  assume that $x, y \in S_E$ satisfies the following conditions.
  \begin{itemize}
  \item $\phi_t(x)\in B_r\cap S_E$ and
  $\phi_{-t}(y)\in B_r\cap S_E\ $ for $\ 0\le t\le t_0$.
  \item The restricted cone fields are not empty.
   Moreover, there exists $a>0$ such that
   $C^{s,\,c}_{E}(\phi_{t_0}(z))\ne
    \emptyset$ for $z\in U_{a}(\phi_{t_0}(x))\cap S_E$, and
    $C^{u,\,c}_E(\phi_{-t_0}(z'))\ne \emptyset$ for
    each $z'\in U_{a}(\phi_{-t_0}(y))\cap S_E$.
  \end{itemize}
  Then there exists at least one  stable curve $\gamma_E^s(x)$ and
  one unstable curve   $\gamma_E^u(y)$.

   If  $a\ge \sqrt{c^2+1}\,r e^{-(\lambda_2-\epsilon)t_0}$, then
   the stable curve $\gamma_E^s(x)$ and the unstable one
   $\gamma_E^u(y)$ can be extended to the boundary of
   $B_r(x)$ and of $B_r(y)$ respectively. Furthermore,
  $$ \|\phi_t(x)- \phi_t(x_1)\|\le e^{-(\lambda_2-\epsilon)t},
  \quad x_1\in \gamma_E^s(x),\ 0\le t\le t_0$$
  and
  $$ \|\phi_{-t}(y)- \phi_{-t}(y_1)\|\le e^{-(\lambda_2-\epsilon)t},
  \quad y_1\in \gamma_E^u(y),\ 0\le t\le t_0. $$
  It is possible to choose the curves to be $C^1$.
\end{prop}
\begin{rmk}
  The stable and unstable curves are not unique. Locally, there exists
  a cone family such that any curve tangent to this cone family is a
  stable/unstable curve.
\end{rmk}

\begin{proof}
  Let us denote $x'=\phi_{t_0}(x)$. From the smoothness of the flow,
  we have that there exist neighborhoods $U$ of $x$ and $U'$ of $x'$
  such that $\phi_{t_0}(U)=U'$ and $\phi_t(U)\in B_r$ for all $0\le
  t\le t_0$. By intersecting $U'$ with $U_a(x')$ if necessary, we may
  assume that $U'\subset U_a(x')$. We have that $C^{s,\,c}_E(z)\ne
  \emptyset$ for all $z\in U'$. It then follows that there exists
  a curve $\gamma^s_E(x')\subset U'$ that is tangent to $C^{s,\,c}_E$. As
  $C^{s,\,c}_E$ is backward invariant with respect to the flow, we have
  that $\phi_{-t}(\gamma^s_E(x'))$ is also tangent to $C^{s,\,c}_E$ for
  $0\le t \le t_0$. Let
  $dist(\gamma^s_E)$ denote the length of the curve
  $\gamma^s_E$   and let $\gamma^s_E(x)=\phi_{-t_0}(\gamma^s(x'))$.
  It follows from the properties of the cone field that
  $$
  dist(\gamma^s_E(x))\ge e^{(\lambda_2-\epsilon)t_0}\ dist(\gamma^s_E(x')).
  $$
  We also remark that from
  the fact that $\gamma^s_E(x)$ is tangent to the cone field
  $C^{s,\,c}_E(x)$, the Euclidean diameter (the largest Euclidean
  distance between two points) of $\gamma^s_E(x)$ is bounded by
  $\frac{1}{\sqrt{c^2+1}}\,dist(\gamma^s_E(x))$ from below and by
  $l(\gamma^s_E(x))$ from above.

  Let $x_1$ be one of the end points of $\gamma^s_E(x)$ and
  $x_1'=\phi_{t_0}(x_1)$. We may apply the same arguments to $x_1$ and
  $x_1'$, and extend the curves $\gamma^s_E(x)$ and $\gamma^s_E(x')$ beyond
  $x_1$ and $x_1'$, unless either $x_1\in \partial B_r$ or
  $x_1'\in \partial U_a(x')$. This extension can be made keeping
  the $C^1$ smoothness of $\gamma$. Denote $\gamma^s_E(x)|[x,x_1]$ the segment
  on $\gamma^s_E(x)$ from $x$ to $x_1$. We have that
$$
\|x_1'- x'\|\le \,dist(\gamma^s_E(x')|[x',x_1'])\le
$$
$$
\le e^{-(\lambda_2-\epsilon)t_0} \,dist(\gamma^s_E(x)|[x,x_1])\le
e^{-(\lambda_2-\epsilon)t_0}  \|x-x_1\|\sqrt{c^2+1}.
$$
It follows that if $a\ge r\sqrt{c^2+1}\,e^{-(\lambda_2-\epsilon)t_0}$,
$x_1$ will always reach boundary of $B_r$ before $x_1'$ reaches the
boundary of $U_a(x')$. This proves that the stable curve can be
extended to the boundary of $B_r$.

The estimate $\|\phi_t(x)- \phi_t(x_1)\|\le e^{-(\lambda_2-\epsilon)t}$
follows directly from the earlier estimate of the arc-length. This
concludes our proof of the proposition for stable curves. The proof
for unstable curves follows from the same argument, but with
$C^{s,\,c}_E$ replaced by $C^{u,\,c}_E$ and $t$ by $-t$.
\end{proof}

In order to apply Proposition~\ref{stable-curve} to the local map, we
need to show that the restricted cone fields are not empty.
(see also the warning after Lemma \ref{local-cone})
\begin{lem}\label{nonemptycone}
  There exists $0<a\le \delta$ and $c>0$ such that for any
  $x=(s,u)\in \Sigma^{s,E}_+$ with $\|u\|\le a$, and $|s_2|\le 2\delta$,
  we have
  $C_{E}^{u,\,c}(x)\ne \emptyset$. Similarly, for any
  $y\in \Sigma^{u,E}_+$  with $|s|\le a$ and $|u_2|\le 2\delta$,
  we have $C^{s,\,c}_{E}(y)\ne \emptyset$.
\end{lem}
\begin{proof}
  We note that
  $$ \nabla N = (\lambda_1 u_1 + u O_1, \lambda_2 u_2 + u O_1,
  \lambda_1 s_1 + s O_1, \lambda_2 s_2 + s O_1),$$
  and hence for small $\|u\|$, $\nabla N \sim (0, 0,
  \lambda_1 s_1, \lambda_2 s_2)$. Since $|s_2|\le 2\delta =2|s_1|$ on
  $\Sigma^s_+$, we have the angle between $\nabla N$ and $u_1$ axis
  is bounded from below. As a consequence, there exists $c>0$,
  such that $C^{u,\,c}$ has nonempty intersection with the tangent
  direction of $S_E$ (which is orthogonal to $\nabla N$). The lemma follows.
\end{proof}

\begin{proof}[Proof of Theorem~\ref{local-domain}]
  We will apply Proposition~\ref{stable-curve} to the pair $x_T^{++}$
  and $y_T^{++}$ which we will denote by $x_T$ and $y_T$ for short.
  Since the curve $\gamma^+$ is tangent to the $s_1$--axis, for $\delta$
  sufficiently small,
  we have $p^+=(\delta, s_2^+, 0, 0)$ satisfies $|s_2|\le \delta$. As
  $x_T\to p^+$, for sufficiently large $T$, we have $x_T=(s_1, s_2, u_1, u_2)$
  satisfy $|u|\le a/2$ and $|s_2|\le 3\delta/2$, where $a$ is as in
  Lemma~\ref{nonemptycone}. As a consequence, for each
  $x' \in U_{a/2}(x_T)\cap \Sigma^{s,E}_+$, we have
  $C_{E}^{u,\,c}(x')\ne \emptyset$. Similarly, we conclude that for each
  $y' \in U_{a/2}(y_T)\cap \Sigma^{u,E}_+$, $C_{E}^{s,\,c}(y')\ne \emptyset$.
  We may choose $T_0$ such that $a/2\ge \sqrt{c^2+1} r e^{-(\lambda_2-\epsilon)T_0}$.

Let $\bar{\gamma}$ be a stable curve containing $x_T$ extended to the boundary
of $B_{r/2}$. Denote the intersection with the boundary $\bar{x}_1$ and
$\bar{x_2}$ and let $\bar{y}_1$ and $\bar{y}_2$ be their images under $\phi_T$.
Let $\gamma_{13}'$ and $\gamma_{24}'$ be unstable curves containing $\bar{y}_1$
and $\bar{y}_2$ extended to the boundary of $B_r$, and let $\gamma_{13}$ and
$\gamma_{24}$ be their preimages under $\phi_T$. Pick $x_1$ and $x_3$ on
the curve $\gamma_{13}$ and let $y_1$ and $y_3$ be their images. It is
possible to pick $x_1$ and $x_3$ such that the segment $y_1 y_3$ on
$\gamma_{13}'$ extends beyond $B_{r/2}$. We now let $\gamma_{12}$ and
$\gamma_{34}$ be stable curves containing $x_1$ and $x_3$ that intersects
$\gamma_{24}$ at $x_2$ and $x_4$.

Note that by construction, $\bar{\gamma}$ and $\gamma_{13}'$ are extended
to the boundary of $B_{r/2}$. As the parameter $T\to \infty$, the limit of
the corresponding curves still extends to the boundary of $B_{r/2}$, which
contains $\gamma_s^+$ and $\gamma_u^+$ respectively. Moreover, by
Proposition~\ref{stable-curve}, the Hausdorff distance between
$\gamma_{12}$, $\gamma_{34}$ and $\bar{\gamma}$ is exponentially
small in $T$, hence they have a common limit. The same can be said
about $\gamma_{13}'$ and $\gamma_{24}'$.

There exists a Poincar\'e map taking $\gamma_{12}$ and $\gamma_{34}$
to curves on the section $\Sigma^s_+$; we abuse notation and still
call them $\gamma_{12}$ and $\gamma_{34}$. Similarly, $\gamma_{13}'$
and $\gamma_{24}'$ can also be mapped to the section $\Sigma^u_+$ by
a Poincar\'e map. These curves on the sections $\Sigma^s_+$ and
$\Sigma^u_+$ completely determines the rectangle
$R^{++}(T)\subset \Sigma^{s,E(T)}_+$. Note that the limiting properties
described in the previous paragraph is unaffected by the Poincar\'e map.
This concludes the proof of Theorem \ref{local-domain}.
\end{proof}


By construction curves $\gm_{12}$ and $\gm_{34}$ can be
selected as stable and $\gm_{14}$ and $\gm_{23}$ ---
as unstable. It leads to the following
\begin{cor}\label{trans-rect}
  There exists $T_0>0$ such that the following hold.
  \begin{enumerate}
  \item For $T\ge T_0$, $\Phg^+ \circ \Phl^{++}(R^{++}(T))$
  intersects $R^{++}(T)$ transversally.
       Moreover, the  images of $\gamma_{13}$ and $\gamma_{24}$ intersect
       $\gamma_{12}$ and $\gamma_{34}$ transversally, and the images of
       $\gamma_{12}$ and $\gamma_{34}$ does not intersect $R^{++}(T)$.
  \item For $T \ge T_0$, $\Phg^-\circ \Phl^{--}(R^{--}(T))$ intersects
  $R^{--}(T)$ transversally.
  \item For $T, T' \ge T_0$ such that $R^{+-}(T)$ and $R^{-+}(T')$ are
  on the same energy surface: $\Phg^- \circ \Phl^{+-} (R^{+-}(T))$
  intersect $R^{-+}(T')$ transversally, and
  $\Phg^+\circ \Phl^{-+}(R^{-+}(T'))$ intersect $R^{+-}(T)$ transversally.
  \end{enumerate}
\end{cor}

\brm \label{different-T}
Later we show that, for fixed $T$, the value $T'$ satisfying condition in the third item is unique.
\erm

\section{Existence of shadowing period orbits
and the proof of Theorem \ref{periodic}}

\subsection{Conley-McGehee isolation blocks}

We will use Theorem~\ref{local-domain} to prove Theorem~\ref{periodic}.
We apply the construction in the previous section to all four local
maps in the neighborhoods of the points $p^\pm$ and $q^\pm$, and obtain
the corresponding rectangles.

For the map $\Phg^+\circ \Phl^{++}|S_{E(T)}$, the rectangle $R^{++}(T)$ is
an isolation block in the sense of Conley and McGehee (\cite{McG73}), defined as follows. 

A rectangle $R = I_1\times I_2\subset \R^d\times \R^k$,
$I_1=\{\|x_1\|\le 1\}$, $I_2=\{\|x_2\| \le 1\}$ is called
{\it an isolation block} for the $C^1$ diffeomorphism $\Phi$, if the following hold:
\begin{enumerate}
\item The projection of $\Phi(R)$ to the first component covers $I_1$.
\item $\Phi|I_1 \times \partial I_2$ is homotopically equivalent to
identity restricted on $I_1 \times (\R^k\setminus \ int\ I_2)$.
\end{enumerate}
If $R$ is an isolation block of $\Phi$, then the set
$$W^+=\{x\in R: \Phi^k(x)\in R,\ k\ge 0\}\quad
 \text{ (resp. }
W^-=\{x\in R: \Phi^{-k}(x)\in R,\ k\ge 0\})
$$
projects onto $I_1$ (resp. onto $I_2$) (see \cite{McG73}). If some
additional cone conditions are satisfied, then $W^+$ and $W^-$ are
in fact $C^1$ graphs. Note that in this case, $W^+\cap W^-$ is the unique
fixed point of $\Phi$ on $R$.

As usual, we denote by $C^{u,c}(x)=\{c\|v_1\| \le \|v_2\|\}$ the unstable cone at $x$.
We denote by $\pi C^{u,c}(x)$ the set $x+ C^{u,c}(x)$, which corresponds to
the projection
of the cone $C^{u,c}(x)$ from the tangent space to the base set. The stable
cones are defined similarly.  Let $U\subset \R^d\times \R^k$ be an open set and
$\Phi:U\to \R^d\times \R^k$ a $C^1$ map.
\begin{enumerate}
\item[C1.] $D\Phi$ preserves the cone field $C^{u,c}(x)$, and there exists
$\Lambda>1$ such that $\|D\Phi(v)\|\ge \Lambda \|v\|$ for any $v\in C^{u,c}(x)$.
\item[C2.] $\Phi$ preserves the projected restricted cone field $\pi C^{u,c}$, i.e.,
for any $x\in U$, 
$$\Phi(U\cap \pi C^{u,c}(x)) \subset C^{u,c}(\Phi(x))\cap \Phi(U).$$
\item[C3.] If $y\in \pi C^{u,c}(x)\cap U$, then $\|\Phi(y)-\Phi(x)\| \ge \Lambda \|y-x\|$.
\end{enumerate}
The unstable cone condition guarantees that any forward invariant set is contained in
a Lipschitz graph.
\begin{prop}[See \cite{McG73}]\label{Lipschitz}
	Assume that $\Phi$ and $U$ satisfies C1-C3, then any forward invariant set
	$W\subset U$ is contained in a Lipschitz graph over $\R^k$ (the stable direction).
\end{prop}
\begin{proof}
We claim that any $x, y\in W$ must satisfy $y\notin \pi C^{u,c}(x)$. Assume otherwise,
then we have $\Phi^k(y) \in \pi C^{u,c}(\Phi^k(x))$ for all $k \ge 0$, and hence
$$\|\Phi^k(y) - \Phi^k(x)\| \ge \Lambda^k \|y-x\|.$$
But this contradicts with $\Phi^k(x), \Phi^k(y) \in U$ for all $k \ge 0$.
It follows that $y \in \pi C^{s,1/c}(x)\cap U$, which implies the Lipschitz condition.
\end{proof}
Similarly, we can define the conditions C1-C3 for the inverse map and the stable cone,
and refer to them as ``stable C1-C3'' conditions.  Note that if $\Phi$ and $U$ satisfies
both the isolation block condition and the stable/unstable cone conditions, then $W^+$
and $W^-$ are transversal Lipschitz graphs. In particular, there exists a unique
intersection, which is the unique fixed point of $\Phi$ on $R$.  We summarize as follows. 

\begin{cor}
  Assume that $\Phi$ and $U$ satisfies the isolation block condition, and that $\Phi$ and $U$ (resp. $\Phi^{-1}$ and $U \cap \Phi(U)$) satisfies the unstable (resp. stable) conditions C1-C3. Then $\Phi$ has a unique fixed point in $U$. 
\end{cor}

\subsection{Single leaf cylinder}

We now apply
the isolation block construction to the maps and rectangles obtained in Corollary~\ref{trans-rect}.

\begin{prop}\label{fixedpt}
  There exists $T_0>0$ such that the following hold.
  \begin{itemize}
  \item For $T\ge T_0$, $\Phg^+\circ \Phl^{++}$ has a unique fixed point $p^+(T)$
  on $\Sigma_s^+\cap R^{++}(T)$;
  \item For $T\ge T_0$, $\Phg^-\circ \Phl^{--}$ has a unique fixed point $p^-(T)$
  on $\Sigma_s^-\cap R^{--}(T)$;
  \item  For $T, T' \ge T_0$ such that $R^{+-}(T)$ and $R^{-+}(T')$ are on
  the same energy surface: $\Phg^+ \circ \Phl^{-+}\circ \Phg^-\circ \Phl^{+-}$
  has a unique fixed point $p^c(T)$ on $R^{+-}(T)\cap (\Phg^-\circ \Phl^{+-})^{-1}(R^{-+}(T'))$.
  \end{itemize}
\end{prop}
Note that in the third case of Proposition~\ref{fixedpt},  it is possible
to choose $T'$ depending on $T$ such that the rectangles are on the same
energy surface, if $T$ is large enough. Moreover, as in remark \ref{different-T}
we later show that such $T'=T'(T)$ is unique.
As a consequence, the fixed point $p^c(T)$ exists for all sufficiently large $T$.

Each of the fixed points $p^+(T)$, $p^-(T)$ and $p^c(T)$ corresponds to a periodic
orbit of the Hamiltonian flow. In addition,
{\it the energy of the orbits are monotone in $T$, and hence we can
switch to $E$ as a parameter.}

\begin{prop}\label{monotone}
The curves $(p^+(T))_{T\ge T_0}$, $(p^-(T))_{T\ge T_0}$ and $(p^{\,c}(T))_{T\ge T_0}$
are $C^1$ graphs over the $u_1$ direction with uniformly bounded derivatives.
Moreover, the energy $E(p^+(T))$, $E(p^-(T))$ and $E(p^{\,c}(T))$ are monotone
functions of $T$.
\end{prop}

We now prove Theorem~\ref{periodic} assuming Propositions~\ref{fixedpt} and \ref{monotone}.

\begin{proof}[Proof of Theorem~\ref{periodic}]
Note that due to Proposition~\ref{bvp}, the sign of $s_1$ and $u_1$ does not change
in the boundary value problem. It follows that the energies of $p^\pm(T)$ are positive,
and the energy of $p^{\,c}(T)$ is negative. Reparametrize by energy, we obtain families
of  fixed points $(p^\pm(E))_{0<E\le E_0}$ and $(p^{\,c}(E))_{-E_0 \le E <0}$, where
$$E_0 = \min\{E(p^+(T_0)), E(p^-(T_0)), -E(p^{\,c}(T_0))\}.$$
 We now denote the full orbits of these fixed points $\gamma^+_E$, $\gamma^-_E$ and
 $\gamma^c_E$, and the theorem follows.
\end{proof}

To prove Proposition~\ref{fixedpt}, we notice that the rectangle $R^{++}(T)$ has
$C^1$ sides, and there exists a $C^1$ change of coordinates turning it to a standard
rectangle. It's easy to see that the isolation block conditions are satisfied for
the following maps and rectangles:
$$
\Phg^+\circ \Phl^{++}\quad \text{and}\quad R^{++}(T),
\qquad
\Phg^-\circ \Phl^{--}\quad \text{and}\quad R^{--}(T),$$
$$ \Phg^+\circ \Phl^{-+} \circ \Phg^- \circ \Phl^{+-}
\quad \text{ and }\quad
(\Phg^- \circ \Phl^{+-})^{-1}R^{-+}(T)\cap R^{+-}(T).
$$
   It suffices to prove the stable and unstable conditions C1-C3 for
   the corresponding return map and rectangles. We will only prove
   the C1-C3 conditions conditions for the unstable cone $C_{E}^{u,\,c}$,
   the map $\Phg^+\circ \Phl^{++}$ and the rectangle $R^{++}(T)$; the proof
   for the other cases  can be obtained by making obvious changes to the case covered.

\begin{lem}\label{local-global-cone}
  There exists $T_0>0$ and $c>0$ such that the following hold. Assume that
  $U\subset \Sigma^s_+\cap B_r$ is a connected open set on which the local
  map $\Phl^{++}$ is defined, and for each $x\in U$,
  $$ \inf\{ t\ge 0: \varphi_t(x)\in \Sigma^u_+\} \ge T_0.$$
 Then the map $D(\Phg^+\circ \Phl^{++})$ preserves the
 non-empty cone field
 $C^{u,\,c}$, and the inverse $D(\Phg^+\circ \Phl^{++})^{-1}$ preserves
 the non-empty $C^{s,\,c}$. Moreover, the projected cones
 $\pi C^{u,\,c}\cap U$ and $\pi C^{s,\,c}\cap V$ are preserved by
 $\Phg^+\circ \Phl^{++}$ and its inverse, where $V=\Phg^+\circ \Phl^{++}(U)$.

 The same set of conclusions hold for the restricted version. Namely,
 we can replace $C^{u,\,c}$ and $C^{s,\,c}$
 with $C^{u,\,c}_E$ and $C^{s,\,c}_E$, and $U$  with $U \cap S_E$.
\end{lem}

Let $x\in U$ and denote $y=\Phl^{++}(x)$. We will first show that
$D\Phl^{++}(x)C^{u,\,c}(x)$ is very close to the strong unstable direction $T^{uu}$.
In general, we expect the unstable cone to contract and get closer to the $T^{uu}$
direction along the flow. The limiting size of the cone depends on how close
the flow is to a linear hyperbolic flow.
We need the following auxiliary Lemma.

Assume that $\phi_t$ is a flow on $\R^d\times \R^k$, and $x_t$ is a trajectory
of the flow. Let $v(t)=(v_1(t), v_2(t))$ be a solution of the variational
equation, i.e. $v(t) = D\phi_t(x_t) v(0)$. Denote the unstable cone
 $C^{u,\,c}=\{\|v_1\|^2< c\|v_2\|^2\}$.

\begin{lem}\label{cone-width} With the above notations assume that there exists
$b_2>0$, $b_1 < b_2$ and $\sigma,\delta>0$ such that the variational equation
  $$\dot{v}(t) =
  \begin{bmatrix}
    A(t) & B(t) \\
    C(t) & D(t)
  \end{bmatrix}
  \begin{bmatrix}
    v_1(t) \\ v_2(t)
  \end{bmatrix}
  $$
  satisfy $A \le b_1 I$ and $D\ge b_2 I$ as quadratic forms, and $\|B\|\le \sigma$,
  $\|C\|\le \delta$.

  Then for any $c>0$ and $\epsilon>0$, there exists $\delta_0>0$ such that
  if $\ 0< \delta, \sigma< \delta_0$, we have
  $$ (D\phi_t)\,C^{u,\,c} \subset C^{u,\beta_t}, \quad
  \beta_t = ce^{-(b_2-b_1-\epsilon)t} + \sigma/(b_2-b_1-\epsilon). $$
\end{lem}
\begin{proof}
Denote $\gamma_0 = c$. The invariance of the cone field is equivalent to
$$ \frac{d}{dt} \left( \beta_t^2\langle v_2(t), v_2(t)\rangle - \langle v_1(t),
v_1(t)\rangle \right)\ge 0.$$
Compute the derivatives using the variational equation,  apply the norm bounds and
the cone condition, we obtain
$$
2\beta_t \left(  \beta_t' + (b_2 - \delta \beta_t -b_1) \beta_t - \sigma \right)\|v_2\|^2
\ge 0.$$
We assume that $\beta_t \le 2\gamma_0$, then for sufficiently small $\delta_0$,
$\delta \beta_t \le \epsilon$. Denote $b_3 = b_2 - b_1 - \epsilon$ and let
$\beta_t$ solve the differential equation
$$ \beta_t' = - b_3 \beta_t + \sigma. $$
It's clear that the inequality is satisfied for our choice of $\beta_t$.
Solve the differential equation for $\beta_t$ and the lemma follows.
\end{proof}

\begin{proof}[Proof of Lemma~\ref{local-global-cone}]
  We will only prove the unstable version. By Assumption 4, there exists $c>0$
  such that $D\Phg^{+}(q^+)T^{uu}(q^+)\subset C^{u,\,c}(p^+)$. Note that as
  $T_0 \to \infty$, the neighborhood $U$ shrinks to  $p^+$ and $V$ shrinks
  to $q^+$. Hence there exists $\beta>0$ and  $T_0>0$ such that
  $D\Phg^+(y)C^{u,\,\beta}(y)\subset C^{u,\,c}$ for all $y\in V$.

  Let $(s,u)(t)_{0\le t\le T}$ be the trajectory from $x$ to $y$.
  By Proposition~\ref{bvp}, we have $\|s\|\le e^{-(\lambda_1-\epsilon)T/2}$
  for all $T/2 \le t\le T$. It follows that the matrix for the  variational
  equation
  \be \label{variational-matrix}
  \begin{bmatrix}
    A(t) & B(t) \\
    C(t) & D(t)
  \end{bmatrix}=
  \begin{bmatrix}
    -\diag\{\lambda_1, \lambda_2\} +O(s) & O(s) \\
    O(u) & \diag\{\lambda_1, \lambda_2\} + O(u)
  \end{bmatrix}
  \ee
  satisfies $A\le -(\lambda_1 -\epsilon)I$, $D\ge (\lambda_1-\epsilon)I$,
  $\|C\| = O(\delta)$ and 
  $\|B\| = O(e^{-(\lambda_1-\epsilon)T/2})$.
  As before $C^{u,\,c}(x) = \{\|v_s\| \le c \|v_u\|\}$,
  Lemma~\ref{cone-width} implies
  $$ D\phi_T(x)C^{u,\,c}(x) \subset C^{u,\beta_T}(y),$$
  where $\beta_T = O(e^{-\lambda'T/2})$ and $\lambda' = \min\{\lambda_2-\lambda_1-\epsilon, \lambda_1-\epsilon\}$. Finally, note
  that $D\phi_T(x)C^{u,\,c}(x)$ and $D\Phl^{++}(x)C^{u,\,c}(x)$ differs
  by the differential of the local Poincar\'e map near $y$. Since near $y$
  we have $|s|=O(e^{-(\lambda_1-\epsilon)T})$, using the equation of motion,
  the Poincar\'e map is exponentially close to identity on the $(s_1,s_2)$
  components, and is exponentially close to a projection to $u_2$ on the
  $(u_1,u_2)$ components. It follows that the cone $C^{u,\,\beta_T}$
  is mapped by the Poincar'e map into a strong unstable cone with
  exponentially small size. In particular, for $T\ge T_0$, we have
  $$ D\Phl^{++}(x)C^{u,\,c}(x) \subset C^{u,\,\beta}(y), $$
and
the first part of the lemma follows. To prove the restricted version
we follow the same arguments.
\end{proof}


Conditions C1-C3 follows, and this concludes the proof of Proposition~\ref{fixedpt}.

\begin{proof}[Proof of Proposition~\ref{monotone}]
Again, we will only treat the case of $p^+(T)$.   Note that
$l^+(p^+):=(p^+(T))_{T\ge T_0}$ is a forward invariant set of
$\Phg^+\circ \Phl^{++}$, and by Lemma~\ref{local-global-cone},
the map $\Phg^+\circ \Phl^{++}$ also preserves the (unrestricted)
strong unstable cone field 
$C^{u,\,c}$. Apply Proposition~\ref{Lipschitz}, we obtain that
$l^+(p^+)$ is contained in a Lipschitz graph
over the $s_1u_1u_2$ direction. Since $l^+(p^+)$ is also backward invariant, and using
the invariance of the strong stable cone fields, we have $l^+(p^+)$ is contained in
a Lipschitz graph over the $s_1u_1s_2$ direction. The intersection of
the two Lipschitz graph is a Lipschitz graph over the $s_1u_1$ direction. Since
$l^+(p^+)\subset \{s_1 =\delta\}$, we conclude that $l^+(p^+)$ is Lipschitz
over $u_1$. Since the fixed point clearly depends smoothly on $T$, $l^+(p^+)$
is a smooth curve. The Lipschitz condition ensures a uniform derivative bound.
This proves the first claim of the proposition. Note that this also implies $u_1$
is a monotone function of $T$.

For the monotonicity, note that all $p^+(T)$ are solutions of the Shil'nikov
boundary value problem. 
By definition $(p^+(T))_{T>T_0}$ belong to $\Sigma^s_+$ and we have $s_1=\dt$.
For all finite $T$ the union of $(p^+(T))_{T>T_0}$ is smooth.
Since $l^+(p^+)$ is a Lipschitz graph over $u_1$ for small $u_1$, we have
that the tangent $(ds_2,du_1,du_2)$ is well-defined and ratios
$\frac{ds_2}{du_1}$ and $\frac{du_2}{du_1}$ are bounded.

Theorem~\ref{tangency} implies that the $s_2$, $u_2$ components are
dominated by the $s_1$, $u_1$ directions, namely, there exist $C>0$
and $\al>0$ such that for components of $p^+(T)$ and all $T>T_0$ we have
$|u_2|\le C|u_1|^\al$.

Using the form of the energy given by Corollary \ref{normal-hamiltonian-form}
its differential has the form
$$
dE(s,u)=(\lb_1+O(s,u))\, s_1 du_1+(\lb_1+O(s,u)) \,u_1 ds_1+
$$
$$+
(\lb_2+O(s,u))\, s_2du_2+(\lb_2+O(s,u))\, u_2ds_2.
$$
On the section $\Sigma_+^s$ differential $ds_1=0$ and coefficients
in front of $ds_2$ can be make arbitrary small.
Therefore, to prove monotonicity of $E(p^+(T))$ in $T$ it suffices
to prove that for any $\tau>0$ there is $T_0>0$ such that for any
$T>T_0$ tangent of $l^+(p^+)$ at $p^+(T)$ satisfies $|\frac{du_2}{du_1}|<\tau$.
Indeed, $(s_1,s_2)(T)\to (\dt,s_2^+)$ as $T\to \infty$.

We prove this using Lemma \ref{cone-width} and the form of the equation
in variations (\ref{variational-matrix}). Suppose $|\frac{du_2}{du_1}|>\tau$
for some $\tau>0$ and arbitrary small $u_1$. If $T_0$ is large enough, then
$T>T_0$ is large enough and $u_1$ is small enough. By Theorem ~\ref{tangency}
we have $|u_2|\le C|u_1|^\al$ so $u_2$ is also small enough. Thus,
we can apply Lemma \ref{cone-width} with $v_1=(s_1,s_2,u_1)$ and
$v_2=u_2$.  It implies that the image of a tangent to $l^+(p^+)$ after
application of $D\Phi_{loc}^{++}$ is mapped into a small unstable cone
$C^{u,\bt}$ with $\bt=(e^{-(\lb_2-\lb_1-\eps)T_0}+O(\dt))/\tau$.
However, the image of $l^+(p^+)$ under $D\Phi_{loc}^{++}$ by
definition is $(q^+(T))_{T\ge T_0}$ and its tangent can't be
in an unstable cone. This is a contradiction.

As a consequence, the energy $E(p^+(T))$ depends monotonically on $u_1$.
Combine with the first part, we have $E(p^+(T))$ depends monotonically on $T$.
\end{proof}

\subsection{Double leaf cylinder}

In the case of the double leaf cylinder, there exist two rectangles $R_1$
and $R_2$, whose images under $\Phg\circ \Phl$ intersect themselves
transversally, providing a ``horseshoe'' type picture.

\begin{prop}\label{double}
  There exists $E_0>0$ such that the following hold:
  \begin{enumerate}
  \item For all $0< E \le E_0$, there exist rectangles
  $R_1(E),R_2(E)\in \Sigma^{s,E}_+$ such that for $i=1, 2$,
  $\Phg^i \circ \Phl^{++} (R_i)$ intersects both $R_1(E)$ and $R_2(E)$ transversally.
  \item Given $\sigma=(\sigma_1, \cdots, \sigma_n)$, there exists a unique fixed point $p^{\sigma}(E)$ of
        $$ \prod_{i=n}^1 \left(\Phg^{\sigma_i} \circ  \Phl^{++} \right)|_{R_{\sigma_i}(E)} $$
        on the set $R_{\sigma_1}(E)$. 
  \item The curve $p^{\sigma}(E)$ is a $C^1$ graph over the $u_1$ component
  with uniformly bounded derivatives. Furthermore, $p^{\sigma}(E)$  approaches
  $p^{\sigma_1}$ and for each $1 \le j\le n-1$, 
  $$ \prod_{i=j}^1 \left(\Phg^{\sigma_i} \circ  \Phl^{++} \right) (p^\sigma(E)) $$
  approaches $p^{\sigma_{j+1}}$ as $E\to 0$. 
  \end{enumerate}
\end{prop}
\begin{rmk}
  The second part of Theorem~\ref{periodic} follows from this proposition.
\end{rmk}
\begin{proof}
 Let $R^{++}(E)$ be the rectangle associated to the local map $\Phl^{++}$ constructed
 in Theorem~\ref{local-domain},  reparametrized in  $E$. Note that for sufficiently
 small $\delta$, the curve $\gamma_s^+$ contains both $p^1$ and $p^2$, and $\gamma_u^+$
 contains both $q^1$ and $q^2$.

Let  $V^1\ni q^1$ and $V^2\ni q^2$ be the domains of $\Phg^1$ and $\Phg^2$, respectively.
It follows from assumption A4a$'$ that $\Phg^1 \gamma_u^+\cap V^1$ intersects $\gamma_s^+$
transversally at $p^i$. By Proposition~\ref{local-domain}, for sufficiently small $E>0$,
$\Phg^1 (\Phl^{++}(R^{++}(E))\cap V_1)$ intersects $R^{++}(E)$ transversally.
Let $Z^1\subset V^1$ be a smaller neighborhood of $q^1$. We can truncate the rectangle
$\Phl^{++}(R^{++}(E))$ by stable curves, and obtain a new rectangle $R_1'(E)$ such that
$$ \Phl^{++}(R^{++}(E)) \cap Z^1 \subset R_1'(E) \subset  \Phl^{++}(R^{++}(E)) \cap V^1.$$
Denote $R_1(E)=(\Phl^{++})^{-1}(R_1'(E))$. The rectangles $R_2(E)$ and
$R_2'(E)$ are defined similarly. For $i=1,2$,  $\Phg^i \circ \Phl^{++} (R_i(E))$
intersects $R^{++}(E)$, and hence $R_i(E)$ transversally. This proves the first statement. 

Let $R^\sigma(E)$ denote the subset of $R_{\sigma_1}(E)$ on which the composition
        $$ \prod_{i=n}^1 \left(\Phg^{\sigma_i} \circ  \Phl^{++} \right)|_{R_{\sigma_i}(E)} $$
is defined. $R^\sigma(E)$ is still a rectangle. The composition map and the rectangle $R^\sigma(E)$ satisfy the isolation block condition and the cone conditions. As a consequence, there exists a unique fixed point. 

The proof of the $C^1$ graph property is similar to that of Proposition~\ref{monotone}.
\end{proof}

\section{Normally hyperbolic cylinder}\label{NHIC-isolating-block}

\subsection{NHIC for the slow mechanical system}
\label{sec:slow-nhic}

In this section we will prove Theorem~\ref{NHIC-mechanical}. Let us first consider
the single leaf case. We will show that the union
$$\mM:=\bigcup_{0< E\le E_0}\gamma_E^+\cup
\bigcup_{0< E\le E_0}\gamma_E^-\cup
\bigcup_{-E_0\le E<0}\gamma_E^{+-}\cup \gamma^+\cup \gamma^-$$
forms a $C^1$ manifold with boundary. Denote
$$l^+(p^+)=\{p^+(E)\}_{0<E\le E_0}, \quad l^+(p^-)=\{p^-(E)\}_{0<E\le E_0},
$$
$l^+(q^+) = \Phl^{++}(l^+(p^+))$ and $l^+(q^-) = \Phl^{--}(l^+(q^-))$.
Note that the superscript of $l$ indicates positive energy instead of
the signature of the homoclinics. We denote
$$
l^-(p^+) = \{p^c(E)\}_{-E_0 \le E <0}
$$
$l^-(q^-) = \Phl^{+-}(l^-(p^+))$, $l^-(p^-) = \Phg^-(l^-(q^-))$
and $l^-(q^+) = \Phl^{-+}(l^-(p^-))$. An illustration of $\mM$
the curves $l^\pm$ are included in Figure~\ref{fig:mnfd}.

\begin{figure}[t]
 \centering
 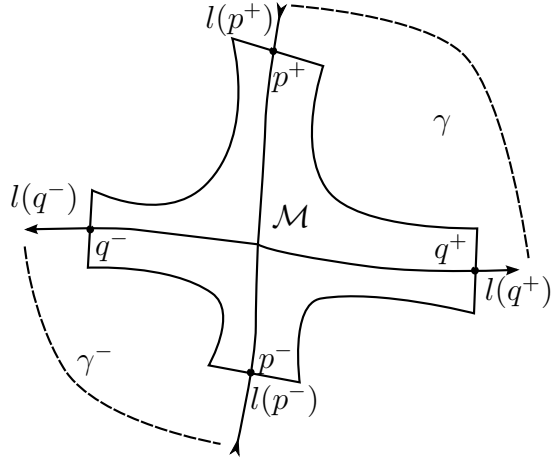
 \caption{Invariant manifold $\mM$ near the origin}
 \label{fig:mnfd}
\end{figure}

 By Proposition~\ref{monotone},
$l^\pm(y)$ ($y$ is either $p^\pm$, or  $q^\pm$) are all $C^1$ curves
with uniformly bounded derivatives, hence they extend to $y$ as
$C^1$ curves. Denote $l(y) = l^+(y)\cup l^-(y) \cup \{y\}$ for
$y$ either $p^\pm$, or  $q^\pm$.

\begin{prop}
  There exists one dimensional subspaces \
  $L(p^\pm)\subset  T_{p^\pm}\Sigma^u_\pm$ and \
  $L(q^\pm)\subset T_{q^\pm}\Sigma^s_\pm$  such that
  the curves $l(p^\pm)$ are tangent to $L(p^\pm)$ at $p^\pm$
  and $l(q^\pm)$ are tangent to $L(q^\pm)$ at $q^\pm$.
\end{prop}
\begin{proof}
  Each point $x\in l(p^+)$ contained in $S_E$ is equal to
  the  exiting position $s(T_E),u(T_E)$ of a solution
  $(s,u):[0,T_E]\to  B_r$ that satisfies Shil'nikov's boundary
  value problem (see Proposition~\ref{bvp}). As $x\to p^+$,
  $E\to 0$ and $T_E\to
  \infty$. According to Corollary~\ref{tangency}, $l(p^+)$ must
  be tangent to the  plane $\{s_1=u_2=0\}$. Similarly,
  $l(q^+)$ must be tangent to the plane $\{u_1=s_2=0\}$.
  On the other hand, due to assumption 4 on the  global map
  (see Section~\ref{sec:intro}), the image of  $D\Phg^+\{u_1=s_2=0\}$
  intersects $\{s_1=u_2=0\}$ at a one  dimensional subspace. Denote
  this space $L(p^+)$ and write  $L(q^+)=D(\Phg^+)^{-1}L(p^+)$.
  Since $l(p^+)$ must be tangent to  both $\{u_2=s_1=0\}$ and
  $D\Phg^+\{u_1=s_2=0\}$, $l(p^+)$ is tangent to $L(p^+)$.
  We also obtain the tangency of $l(q^+)$ to $L(q^+)$  using
  $l(q^+)=(\Phg^+)^{-1}l(p^+)$. The case for $l(p^-)$ and
  $l(p^-)$ can be proved similarly.
\end{proof}

We have the following continuous version of Lemma~\ref{local-global-cone},
which states that the flow on  $\mM$ preserves the strong stable and
strong unstable cone fields. The proof of Lemma~\ref{mM-cone} is
contained in the proof of Lemma~\ref{local-global-cone}.

\begin{lem}\label{mM-cone}
  There exists $c>0$ and $E_0>0$ and  continuous cone family $C^{u}(x)$
  and $C^{s}(x)$, such that for all $x\in \mM$, the following hold:
  \begin{enumerate}
  \item $C^{s}$ and $C^{u}$ are transversal to $T\mM$, $C^{s}$ is
    backward invariant and $C^{u}$ is forward invariant.
  \item There exists $C>0$ such that the following hold:
    \begin{itemize}
    \item $\|D\phi_t(x)v\|\ge C e^{(\lambda_2-\epsilon)t}$, $v\in
      C^{u}(x)$, $t\ge 0$;
    \item $\|D\phi_t(x)v\|\ge C e^{-(\lambda_2-\epsilon)t}$, $v\in
      C^{s}(x)$, $t\le 0$.
    \end{itemize}
  \item There exists a neighborhood $U$ of $\mM$ on which the projected
  cones $\pi C^{u}\cap U$ and $\pi C^{s}\cap U$ are preserved.
  \end{enumerate}
\end{lem}

Note that a continuous version of Proposition~\ref{Lipschitz} also holds.
As a consequence, the the set $\mM$ is contained in a Lipschitz graph over
the $s_1$ and $u_1$ direction. This implies that $\mM$ is a $C^1$ manifold.

\begin{cor}
  The manifold $\mM$ is a $C^1$ manifold with boundaries  $\gamma_{E_0}^+$,
  $\gamma_{E_0}^-$ and  $\gamma_{-E_0}^{+-}$.
\end{cor}
\begin{proof}
  The curves $l(p^\pm)$ and $l(q^\pm)$ sweep out the set $\mM \setminus \{0\}$
  under the flow.  It follows that $\mM$ is smooth at everywhere except may
  be $\{0\}$.  Since any $x\in \mM\cap B_r(0)$  is contained in a solution of
  the Shil'nikov boundary value problem, Corollary~\ref{tangency} implies that
  $x$ is  contained in the set $\{|s_2|\le C|s_1|^\alpha, |u_2|\le  C|u_2|^\alpha\}$.
  It follows that the tangent plane of $\mM$ to $x$ converges to the plane
  $\{s_2=u_2=0\}$ as $(s,u)\to 0$.
\end{proof}

\begin{cor}\label{nhm}
  There exists a invariant splitting $E^s \oplus T\mM \oplus E^u$ and
  $C>0$ such that the following hold:
  \begin{itemize}
  \item $\|D\phi_t(x)v\|\ge C e^{(\lambda_2-\epsilon)t}$, $v\in
    E^u(x)$, $t\ge 0$;
  \item $\|D\phi_t(x)v\|\ge C e^{-(\lambda_2-\epsilon)t}$, $v\in
    E^s(x)$, $t\le 0$;
  \item $\|D\phi_t(x)v\|\le C e^{(\lambda_1+\epsilon)|t|}$, $v\in
    T_x\mM$, $t\in \R$.
  \end{itemize}
\end{cor}
\begin{proof}
  The existence of $E^s$ and $E^u$, and the expansion/contraction properties
  follows from  standard hyperbolic arguments, see \cite{HPS77}, for example.
  We now prove that third statement. Denote $v(t)=D\phi_t(x)v$ for $v\in T_x\mM$.
  Decompose $v(t)$ into $(v_{s_1}, v_{s_2}, v_{u_1}, v_{u_2})$, we have
  $\|(v_{s_1}, v_{u_1})(t)\|\le C e^{(\lambda_1 + \epsilon)|t|}$. However,
  since $\mM$ is a Lipschitz graph over $(s_1, u_1)$, the $(v_{s_2}, v_{u_2})$
  components are bounded uniformly by the $(v_{s_1}, v_{u_1})$ components.
  The norm estimate follows.
\end{proof}

\begin{rmk}
Part 1 of Theorem~\ref{NHIC-mechanical} follows from the last two corollaries.
\end{rmk}

We now come to the double leaf case. Denote
$l(p^1)= \bigcup_{e\le E\le E_0}p^\sigma(E)$, where $p^\sigma(E)$ is the fixed point in Proposition~\ref{double}.  We have that $l(p^{\sigma_1})$
sweeps out $\mM_h^{e,E_0}$ in \emph{finite} time. As a consequence
$\mM_h^{e, E_0}$ is a $C^1$ manifold.  Similar to Lemma~\ref{mM-cone},
the flow on $\mM_h^{e, E_0}$ also preserves the strong stable/unstable
cone fields. The fact that $\mM_h^{e, E_0}$ is normally hyperbolic follows
from the invariance of the cone fields, using the same proof as that of
Corollary~\ref{nhm}. This concludes the proof the Theorem~\ref{NHIC-mechanical}, part 2.


\subsection{Derivation of the slow mechanical system }\label{slow-fast}

We denote by $p_0$ the intersection of the resonance $\Gamma_{\vec k}$ and $\Gamma_{\vec k'}$. This means
$$ \vec k_1 \cdot \partial_pH(p_0) + k_0 =0, \quad \vec k_1' \cdot \partial_pH(p_0) + k_0' =0. $$
We consider the autonomous version of the system  $H_\epsilon(\theta, p, t, E) = H_0 + \epsilon H_1(\theta, p, t) + E$.  In the $\sqrt{\epsilon}$ neighborhood of $p_0$, we have the following the normal form
$$ H_\epsilon(\theta, p, t, E)  = H_0(p) + \epsilon Z(\vec k_1 \cdot \theta + k_0, \vec k_1' \cdot \theta + k_0', p)  + \epsilon R  + E,$$
where $\|R\|_{C^2} = O(\epsilon)$. Denote $\theta^{ss} = \vec k_1 \cdot \theta + k_0$ and $\theta^{sf}= \vec k_1' \cdot \theta + k_0'$ and $\theta^s = (\theta^{ss}, \theta^{sf})$, we further write
\begin{multline*}
  H_\epsilon(\theta, p, t, E) = H_0(p_0) +  \partial H_0(p_0) \cdot (p-p_0) + E + \langle \partial^2_{pp} H_0(p_0) (p-p_0), p-p_0\rangle \\
  + \epsilon Z(\theta^s, p_0)  + \epsilon R',
\end{multline*}
where $R' = R + Z(\theta^s,p) - Z(\theta^s, p_0) + \frac1{\epsilon}O(|p-p_0|^3)$. 
We make a symplectic coordinate change  $(\theta, p, t, E) \to (\theta^s, p^s, t, E')$  by taking
$$
\begin{bmatrix}
  p \\ E
\end{bmatrix} =
\begin{bmatrix}
  B^T & 0 \\
  k_0, k_0' & 1
\end{bmatrix}
\begin{bmatrix}
  p^s \\ E'
\end{bmatrix},\quad  \text{ where } B=
\begin{bmatrix}
\vec k_1 \\ \vec k_1'
\end{bmatrix}.
$$
Denote $p^s_0 = (B^T)^{-1} p_0$, we have
\begin{multline*} \partial_p H_0(p_0) (p-p_0) + E = \partial_p H_0(p_0) B^T (p^s-p^s_0) + k_0 p^{ss} + k_0' p^{sf} + E' \\
  = (\partial_pH_0(p_0) \cdot k_1 + k_0)(p^{ss}-p_0^{ss}) +  (\partial_pH_0(p_0) \cdot k_1' + k_0')(p^{sf}-p_0^{sf})  + (k_0, k_0')\cdot (p^{ss}_0, p^{sf}_0) \\
  = (k_0, k_0')\cdot (p^{ss}_0, p^{sf}_0),
\end{multline*}
hence
\begin{multline*} H_\epsilon(\theta^s, p^s, t, E') = H_0(p_0)  +  (k_0, k_0')\cdot (p^{ss}_0, p^{sf}_0)  + E'\\
   + \langle B\partial^2_{pp}H_0(p_0)B^T (p^s - p_0^s), p^s - p_0^s\rangle + \epsilon Z(\theta^s, p_0) +  \epsilon R'.
\end{multline*}
Denote $ I^s = (p^s-p^s_0)/\sqrt{\epsilon}$, 
\be \label{slow-kinetic}
K(I^s) = \langle B\partial^2_{pp}H_0(p_0)B^T I^s, I^s \rangle,
\ee
\be
\label{slow-potential}
U(\theta^s) = -Z(\theta^s, p_0).
\ee
The flow of $H_\epsilon(\theta^s, p^s, t)$ is conjugate to the flow of the rescaled Hamiltonian
\begin{equation}  \label{eq:slow-system} \frac{1}{\sqrt{\epsilon}} H_\epsilon(\theta^s, \sqrt{\epsilon} I^s, t) = c_0/\sqrt{\epsilon} + \sqrt\epsilon (K(I^s) -  U(\theta^s)) + \sqrt\epsilon R'(\theta^s, \sqrt{\epsilon} I^s,t ),
\end{equation}
where $c_0 = H_0(p_0) + (k_0, k_0')\cdot (B^T)^{-1}p_0$.
By a direct computation, we have the $C^2$ norm of $R'(\cdot, \sqrt{\epsilon}\, \cdot, \cdot)$ is bounded by $O(\sqrt\epsilon)$. 

\subsection{Normally hyperbolic manifold for double resonance}
\label{sec:gen-nhic}

We now prove Corollary~\ref{existence-NHIC}. By (\ref{eq:slow-system}), our Hamiltonian system is locally equivalent to 
$$ H^s_\epsilon(\theta^s, I^s, t) = K(p) - U(\theta) + O(\sqrt{\epsilon}). $$
 For $\epsilon=0$, the system
$H_0^s$ admits a normally hyperbolic manifold $\mM \times \T$. Moreover,
all conclusions of Corollary~\ref{nhm} carries over to this system. It
is well known that a compact normally hyperbolic manifold without boundary survives small perturbations (see \cite{HPS77}, for example). For manifolds with boundary, we can smooth out the perturbation near the boundary, so that the perturbation preserves the boundary (see \cite{BKZ}, Proposition B.3). This produces a weakly invariant NHIC, in the sense that any invariant set near $\mM \times \T$ and away from the boundary must be contained in the NHIC.   

This concludes the proof of Corollary~\ref{existence-NHIC}.

\appendix

\section{Formulation of the results (intermediate energies)}
\label{sec:intermediate}

Consider the slow mechanical system
$H^s(p^s,\th^s)=K(p^s)-U(\th^s),\ U(\th)\ge 0,\ U(0)=0$
as in (\ref{slow-mechanical}) and $E_0>0$ is small. For
each non-negative energy surface $S_E=\{H^s=E\}$
consider the Jacobi metric $\rho_E(\th)=2(E+U(\th))K$
as defined in (\ref{Jacobi-metric}). Orbits of $H^s$
restricted on $S_E$ are reparametrized geodesics of
$\rho_E$. Fix a homology class $h\in H_1(\T^s,\Z)$.
In the same way as in \cite{Ma1} impose the following assumptions:

\begin{itemize}
 \item [B1.] For each $E > E_0$, each shortest closed geodesic
 $\gm^h_E$ of $\rho_E$ in the homology class $h$ is nondegenerate
 in the sense of Morse.

 \item [B2.] For each $E > E_0$, there are at most two shortest closed
 geodesics of $\rho_E$ in the homology class $h$.

Let $E^* > E_0$ be such that there are two shortest geodesics $\gm^h_{E^*}$
and $\overline \gm^h_{E^*}$ of $\rho_{E^*}$ in the homology class $h$.
Due to non-degeneracy there is local continuation of $\gm^h_{E^*}$ and
$\overline \gm^h_{E^*}$ to locally shortest geodesics $\gm^h_E$ and
$\overline \gm^h_E$. For a smooth closed curve $\gm$ denote by
$\ell_E(\gm)$ its $\rho_E$-length.

\item [B3.] Suppose
\[
\dfrac{d(\ell_E(\gm^h_E))}{dE}|_{E=E^*} \ne
\dfrac{d(\ell_E(\overline \gm^h_E))}{dE} |_{E=E^*}.
\]
\end{itemize}

\blm There is an open dense set of smooth mechanical systems
with properties B1-B3.
\elm

It follows from condition B3 that there are only finitely many
values $\{E_j\}_{j=1}^N$ where there are two minimal geodesics
$\gm^h_{E}$ and $\overline \gm^h_{E}$. To fit
boundary conditions we have $E_0^{-1}=E_{N+1}$. There is $\dt>0$
such that for any $j=1,\dots,N$ the unique shortest geodesic
$\gm^h_E$ has a smooth continuation
$\gm_E^h$ for $E\in [E_j-\dt,E_{j+1}+\dt]$.

Consider the union
\[
 \sM^h_j=\cup_{E\in [E_j-\dt,E_{j+1}+\dt]} \gm^h_E.
\]
It follows from Morse non-degeneracy of $\gm^h_E$ that
$\sM^h_j$ is a NHIC. In the same way as we prove Corollary
 we can prove

\begin{cor} \label{existence-NHIC-bis} For each $j=1,\dots,N$
the system $H_\eps$ has a normally hyperbolic
manifold $\sM^h_{j,\eps}$
which is weakly invariant, i.e. the Hamiltonian vector field
of $H_\eps$ is tangent to $\sM^h_{j,\eps}$.
Moreover, the intersection of $\sM^h_{j,\eps}$  with
the regions  $\{E_j-\dt\le H^s \le E_{j+1}+\dt\}\times \T$
is a graph over $\sM^h_j$.
\end{cor}

Proof of Corollary \ref{existence-NHIC-bis} is very similar
to the proof of Corollary \ref{existence-NHIC}.
Notice that the NHIC is $3$-dimensional.  It has one-dimensional
stable and one-dimensional unstable direction. Consider a box
neighborhood at each point on $\sM^h_{j,\eps}$ formed by taking
$\sg$-box in stable/unstable directions. Taking $\eps$ small
we can make sure that the time-periodic system $H_\eps$ satisfies
isolating block property.

\section{Non-self-intersecting curves on the torus}
\label{sec:homology}

We prove Lemma~\ref{homotopy} in this section. 

Denote $\gamma_1=\gamma^{h_1}_0$ and $\gamma_2=\gamma^{h_2}_0$ and $\gamma=\gamma_h^0$. Recall that $\gamma$ has homology class $n_1 h_1 + n_2 h_2$ and is the concatenation of $n_1$ copies of $\gamma_1$ and $n_2$ copies of $\gamma_2$. Since $h_1$ and $h_2$ generates $H_1(\T^2, \Z)$, by introducing a linear change of coordinates, we may assume $h_1=(1,0)$ and $h_2=(0,1)$. 

Given $y\in \T^2 \setminus \gamma\cup \gamma_1 \cup \gamma_2$, the fundamental group of $\T^2\setminus \{y\}$ is a free group of two generators, and in particular, we can choose $\gamma_1$ and $\gamma_2$ as generators. (We use the same notations for the closed curves $\gamma_i$, $i=1,2$ and their homotopy classes). The curve $\gamma$ determines an element 
$$\gamma = \prod_{i=1}^n \gamma_{\sigma_i}^{s_i}, \quad \sigma_i \in \{1,2\}, \, s_i \in \{0,1\} $$
of this group. Moreover, the translation $\gamma_t(\cdot):=\gamma(\cdot +t)$ of $\gamma$ determines a new element by cyclic translation, i.e.,
$$ \gamma_t = \prod_{i=1}^n \gamma_{\sigma_{i+m}}^{s_{i+m}}, \quad m\in \Z,$$
where the sequences $\sigma_i$ and $s_i$ are extended periodically. We claim the following: 

There exists a unique (up to translation) periodic sequence $\sigma_i$ such that $\gamma=\prod_{i=1}^n \gamma_{\sigma_{i+m}}$ for some $m\in \Z$, independent of the choice of $y$. Note that in particular, all $s_i=1$. 

The proof of this claim is split into two steps. 

\emph{Step 1.} Let $\gamma_{n_1/n_2}(t)=\{\gamma(0)+(n_1/n_2, 1)t, \, t\in \R\}$. We will show that $\gamma$ is isotropic (homotopic along non-self-intersecting curves) to $\gamma_{n_1/n_2}$. To see this, we lift both curves to the universal cover with the notations $\tilde{\gamma}$ and $\tilde\gamma_{n_1/n_2}$. Let $p.q\in \Z$ be such that $ pn_1 -qn_2 =1$ and define  
$$ T\tilde\gamma(t) = \tilde\gamma(t) + (p,q). $$
As $T$ generates all integer translations of $\tilde\gamma$, $\gamma$ is non-self-intersecting if and only if $T\tilde\gamma\cap \tilde\gamma = \emptyset$. Define the homotopy $\tilde\gamma_\lambda = \lambda \tilde\gamma +(1-\lambda) \tilde\gamma_{n_1/n_2}$, it suffices to prove $T\tilde\gamma_\lambda \cap \tilde\gamma_\lambda = \emptyset$. Take an additional coordinate change 
$$
\begin{bmatrix}
  x \\ y
\end{bmatrix} \mapsto
\begin{bmatrix}
  n_1 & p \\ n_2 & q
\end{bmatrix}^{-1}
\begin{bmatrix}
  x \\ y
\end{bmatrix}, 
$$
then under the new coordinates $T\tilde\gamma(t) = \tilde\gamma(t) +(1,0)$. 

Under the new coordinates, $T\tilde\gamma\cap \tilde\gamma = \emptyset$ if and only if any two points on the same horizontal line has distance less than $1$. The same property carries over to $\tilde\gamma_\lambda$ for $0\le \lambda <1$, hence $T\tilde\gamma_\lambda \cap \tilde\gamma_\lambda = \emptyset$.

\emph{Step 2.} By step 1, it suffices to prove that  $\gamma = \gamma_{n_1/n_2}$ defines  unique sequences $\sigma_i$ and $s_i$. Since $\tilde\gamma_{n_1/n_2}$ is increasing in both coordinates, we have $s_i=1$ for all $i$. Moreover, choosing a different $y$ is equivalent to shifting the generators $\gamma_1$ and $\gamma_2$. Since the translation of the generators is homotopic to identity, the homotopy class is not affected. This concludes the proof of Lemma~\ref{homotopy}.

\subsection*{Acknowledgment}

 {\it The first author is partially
supported by NSF grant DMS-1101510.
 The second author wishes to thank 
the Fields Institute program ``transport and disordered system'', where part of the work was carried out. The authors are grateful
to John Mather for several inspiring discussions.
The course of lectures \cite{Ma3} he gave was very helpful
for the authors.}

\def\cprime{$'$}

 \end{document}